\documentclass{amsart}
\usepackage{amssymb}

\newcommand{\map}[1]{\xrightarrow{#1}}

\newcommand{\mil}{\varprojlim}
\newcommand{\dlim}{\varinjlim}
\newcommand{\iso}{\cong}
\newcommand{\define}{\stackrel{\mathrm{def}}{=}}

\newcommand{\Gal}{\mathrm{Gal}}
\newcommand{\Hom}{\mathrm{Hom}}

\newcommand{\alg}{\mathrm{alg}}

\newcommand{\N}{\mathbf{N}}

\newcommand{\Q}{\mathbf Q}
\newcommand{\Z}{\mathbf Z}

\newcommand{\C}{\mathbf C}

\newcommand{\co}{\mathcal {O}}
\newcommand{\Frob}{\mathrm {Frob}}

\newcommand{\gm}{\mathfrak m}
\newcommand{\gn}{\mathfrak n}
\newcommand{\gl}{\mathfrak l}
\newcommand{\ga}{\mathfrak a}
\newcommand{\gb}{\mathfrak b}
\newcommand{\gq}{\mathfrak q}

\newcommand{\unr}{\mathrm{unr}}
\newcommand{\ord}{\mathrm {ord}}
\newcommand{\cl}{\mathfrak L}
\newcommand{\cn}{\mathfrak N}
\newcommand{\sel}{\mathcal F}
\newcommand{\Sel}{\mathrm {Sel}}
\newcommand{\Stub}{\mathrm {Stub}}
\newcommand{\ind}{\mathrm {ind}}

\newcommand{\even}{\mathrm {even}}
\newcommand{\odd}{\mathrm {odd}}

\newcommand{\xgraph}{\mathcal X}
\newcommand{\tbar}{\overline T}

\newcommand{\len}{\mathrm {length}}
\newcommand{\loc}{\mathrm {loc}}
\newcommand{\Tw}{\mathrm {Tw}}

\newcommand{\esodd}{\kappa}
\newcommand{\eseven}{\lambda}

\newcommand{\gp}{\mathfrak P}

\newcommand{\Fil}{\mathrm {Fil}}
\newcommand{\definite}{\mathrm {definite}}
\newcommand{\indefinite}{\mathrm {indefinite}}

\input xy
\xyoption{all}
\begin{document}

\author{Benjamin Howard}
\thanks{This research was supported by an NSF postdoctoral fellowship.}
\address{Department of Mathematics, Harvard University, Cambridge, MA.}
\curraddr{Department of Mathematics, University of Chicago, 
5734 S. University Ave., Chicago, IL 60637}
\email{howard@math.uchicago.edu}

\date{\today}

\title{Bipartite Euler Systems}

\begin{abstract}
If $E$ is an elliptic curve over $\Q$ and $K$ is an imaginary quadratic field,
there is an Iwasawa main conjecture predicting the behavior of 
the Selmer group of $E$ over the anticyclotomic $\Z_p$-extension of
$K$.  The main conjecture takes different forms depending on the 
sign of the functional equation of $L(E/K,s)$.  In the present work
we combine ideas of Bertolini and Darmon with those of Mazur and Rubin
to shown that the 
main conjecture, regardless of the sign of the functional equation, 
can be reduced to 
proving the nonvanishing of sufficiently many $p$-adic $L$-functions
attached to a family of congruent modular forms. 
\end{abstract}

\maketitle

\theoremstyle{plain}
\newtheorem{Thm}{Theorem}[subsection]
\newtheorem{Prop}[Thm]{Proposition}
\newtheorem{Lem}[Thm]{Lemma}
\newtheorem{Cor}[Thm]{Corollary}
\newtheorem{Con}[Thm]{Conjecture}
\newtheorem{BigTheorem}{Theorem}

\theoremstyle{definition}
\newtheorem{Def}[Thm]{Definition}
\newtheorem{Hyp}[Thm]{Hypothesis}

\theoremstyle{remark}
\newtheorem{Rem}[Thm]{Remark}
\newtheorem{Ques}[Thm]{Question}

\renewcommand{\labelenumi}{(\alph{enumi})}
\renewcommand{\theBigTheorem}{\Alph{BigTheorem}}
%\setcounter{section}{-1}

%%%%%%%%%%%%%%%%%%%%%%%%%%%%%%%%%%%%%%%%%%%%%%%%%%%%%%%%%%%%%%%%%%%%%%%%
%%%%%%%%%%%%%%%%%%%%%%%%%%%%%%%%%%%%%%%%%%%%%%%%%%%%%%%%%%%%%%%%%%%%%%%%

\section{Introduction}

%%%%%%%%%%%%%%%%%%%%%%%%%%%%%%%%%%%%%%%%%%%%%%%%%%%%%%%%%%%%%%%%%%%%%%%%%
%%%%%%%%%%%%%%%%%%%%%%%%%%%%%%%%%%%%%%%%%%%%%%%%%%%%%%%%%%%%%%%%%%%%%%%%%

Let $E$ be an elliptic curve over $\Q$ with conductor $N$, and let
$K$ be a quadratic imaginary field of discriminant $d_K$ prime to $N$
with quadratic character $\epsilon$.  By work of Gross and Zagier,
the sign of the functional equation of $L(E/K,s)$ is equal to
$-\epsilon(N)$,
For a rational prime $p\nmid 6 d_KN$ at which $E$ has good \emph{ordinary} 
reduction, let $D_\infty/K$ be the anticyclotomic
$\Z_p$-extension of $K$.  The behavior of the Selmer group 
$\Sel_{p^\infty}(E/D_\infty)$ depends crucially on the value 
of $\epsilon(N)$: if $\epsilon(N)=1$ then it is conjectured that
the Pontryagin dual, $X$, of $\Sel_{p^\infty}(E/D_\infty)$ has rank
one over the Iwasawa algebra $\Lambda=\Z_p[[\Gal(D_\infty)/K]]$
and that the characteristic ideal of the torsion submodule can be expressed
in terms of Heegner points arising from a Shimura curve parametrization of $E$.
We call this the indefinite case.  In the definite case, $\epsilon(N)=-1$,
it is conjectured that $X$ is a torsion $\Lambda$-module with characteristic
ideal given by a $p$-adic $L$-function.  
We refer to these two conjectures collectively as the Iwasawa main conjecture.

Much is known about the Iwasawa main conjecture,
see \cite{bertolini, me, me2} for the indefinite case and \cite{BD03}
for the definite case.  In particular, in either case one knows that the rank
of $X$ is as predicted above, and one knows one divisibility of 
the conjectured equality for the characteristic
ideal of the torsion submodule (in the indefinite case this is conditional
on as yet unpublished work of Cornut and Vatsal generalizing the main 
result of \cite{cornut} to
Heegner points on Shimura curves attached to indefinite quaternion algebras, 
see the main results of \cite{me2}).  
The goal of the present article is to demonstrate that the methods used by 
Bertolini and Darmon \cite{BD03} to treat the definite case
can be used give a uniform treatment of the two cases,
and to develop a criterion to determine when the known divisibility
is actually an equality. We do this by developing a theory of 
\emph{bipartite Euler systems} similar in spirit to Mazur and Rubin's
\cite{mazur-rubin} theory of Kolyvagin systems, but adapted to fit the
family of cohomology classes constructed by Bertolini and Darmon. 

Bertolini and Darmon's construction is, roughly speaking, as follows.  Let
$f$ be the modular form of level $N$ attached to $E$.  For a choice
of positive integer $k$ one can define a set of \emph{admissible} primes
$\cl_k$, all of which are inert in $K$, with the property that for 
any $n\in\cn_k$ (the set of squarefree products of primes in $\cl_k$)
there is a modular form $f_n$ of level $nN$ which is congruent 
to $f$ modulo $p^k$.   This modular form comes to us via a generalization
of Ribet's well-known level raising theorem. Define a
graph whose vertices are the elements of $\cn_k$ with edges connecting
$n$ to $n\ell$ for each coprime $n\in\cn_k$ and $\ell\in\cl_k$.  
A vertex $n$ is said to be either definite or indefinite depending on 
whether $\epsilon(nN)$ is $-1$ or $1$, respectively, and this defines 
a bipartition of the graph: every edge connects a definite vertex 
to an indefinite vertex.
At an indefinite vertex the modular form
$f_n$ allows one to define a cohomology class 
$\esodd_n\in \mil_r H^1(D_r/K, E[p^k])$, which arises as the Kummer image
of Heegner points on the abelian variety attached to $f_n$.  
At a definite vertex one can attach to $f_n$ a $p$-adic
$L$-function $\eseven_n\in\Lambda/p^k\Lambda$. 
There are reciprocity laws relating the elements at any two adjacent
vertices;  these reciprocity laws are examples of Jochnowitz
congruences in the sense of \cite{BD99b}. 
The \emph{pair} of families 
$$
\{\esodd_n\mid n\in\cn_k, n \mathrm{\ indefinite}\}
\hspace{1cm} 
\{\eseven_n\mid n\in\cn_k,  n \mathrm{\ definite} \}
$$ 
is then our prototype of a bipartite Euler system.
Our main result asserts that the existence of a bipartite Euler system
implies one divisibility of the Iwasawa main conjecture, and 
if one can prove sufficiently many nonvanishing
theorems for the $p$-adic $L$-functions $\eseven_n$ then equality holds in the
Iwasawa main conjecture.  To emphasize, this approach treats the definite
and indefinite cases on completely equal footing.
The precise statement is given in Theorem \ref{abstract mc}.
 
The reader is referred to \cite{BD03} for the details of the construction
sketched above.  In the present article we simply assume that a pair of
families satisfying the appropriate axioms is given.  It should be noted
that Bertolini and Darmon do not construct enough 
classes to provide an Euler system in our sense.  Those authors
assume that $\epsilon(N)=-1$ and that $f$ is $p$-isolated 
\cite[Definition 1.2]{BD03},
and then choose a particular path (starting at the vertex corresponding to 
the empty product $1$) in the graph defined above.  
The Euler system elements are then constructed
\emph{only at vertices along that path}.  
The path is not allowed to be arbitrary: it is required 
that the modular form $f_n$ is again $p$-isolated at 
each definite vertex in the path.  It would thus be necessary to remove
the $p$-isolated hypothesis in order to make full use of the theory
developed herein.  

Recently Darmon and Iovita \cite{DarIov} have adapted the methods of 
\cite{BD03} to the case where $\epsilon(N)=-1$ and $p$ is a prime
of \emph{supersingular} reduction for $E$.  Given the results of the present
article, it seems likely that these ideas can be pushed further to
cover all four cases (definite/ordinary, indefinite/ordinary, 
definite/supersingular, and indefinite/supersingular; the final case being the
least well understood).  The main (only?) obstruction to doing so is the 
removal of the technical $p$-isolated hypothesis referred to above.

We remark that the idea that Euler systems can be used not only
to bound Selmer groups, but also to give a criterion for the sharpness of
the bound, goes back to Kolyvagin.  This was extended to the Iwasawa-theoretic
setting by Mazur and Rubin \cite{mazur-rubin}, but the criterion for
equality seems very difficult to verify in practice.  In the
usual theory of Euler systems, in e.g. \cite{rubin}, one begins
with cohomology classes (related in some way to $L$-functions)
defined over abelian extensions of the ground 
field $K$, and then applies Kolyvagin's derivative operators
to these classes to obtain classes defined over $K$ itself. These derived
classes are  the \emph{Kolyvagin system}, and are somewhat less directly
related to $L$-functions than the Euler system from which they are derived.  
The criterion for equality (e.g. the \emph{primitivity} of \cite{mazur-rubin}
Definitions 4.5.5 and 5.3.9)
is then a nonvanishing statement
for the Kolyvagin system, rather than for the Euler system itself.   
The observation that Bertolini and Darmon's
methods make no use of Kolyvagin's derivative operators is what allows
us to obtain a criterion for equality in the main conjecture
directly in terms of ($p$-adic) $L$-functions.

Finally, and somewhat more speculatively, we address the question of
whether there exist bipartite Euler systems other than that constructed by
Bertolini and Darmon.  Gross and Kudla \cite{GroKud} 
have investigated the Rankin
triple product $L$-function $L(f\times g\times h,s)$ 
associated to three newforms $f,g,h$ of weight $2$ on $\Gamma_0(N)$.  This 
$L$-function has analytic continuation and functional equation
in $s\mapsto 4-s$, and the sign in the functional equation is given 
by a simple formula.  When this sign is $1$, Gross and Kudla prove
a special value formula similar to Gross's special value formula \cite{Gro85}
in the Heegner point situation, a key ingredient in the reciprocity 
laws used in \cite{BD03}.  When the sign in the functional equation is
$-1$, Gross and Kudla construct a special homologically trivial 
cycle in the codimension $1$ Chow group of a triple product of Shimura
curves.  Applying the $p$-adic Abel-Jacobi map to this special 
cycle yields a class in the Galois cohomology of the tensor product 
$V_f\otimes V_g\otimes V_h$ of the $p$-adic Galois representations
attached to $f,g,h$.  Thus we have the beginnings of a bipartite Euler
system for $V_f\otimes V_g\otimes V_h$.  Moreover, Gross and Kudla 
conjecture that the height of their special cycle in the Chow group
is related to the derivative $L'(f\times g\times h,2)$, in close
analogy with the Gross-Zagier formula.

%%%%%%%%%%%%%%%%%%%%%%%%%%%%%%%%%%%%%%%%%%%%%%%%%%%%%%%%%%%%%%%%%%%%%%%%%
%%%%%%%%%%%%%%%%%%%%%%%%%%%%%%%%%%%%%%%%%%%%%%%%%%%%%%%%%%%%%%%%%%%%%%%%%

\section{Euler Systems over Artinian rings}
\label{S:Euler Systems}
 
%%%%%%%%%%%%%%%%%%%%%%%%%%%%%%%%%%%%%%%%%%%%%%%%%%%%%%%%%%%%%%%%%%%%%%%%
%%%%%%%%%%%%%%%%%%%%%%%%%%%%%%%%%%%%%%%%%%%%%%%%%%%%%%%%%%%%%%%%%%%%%%%%%

In this section we develop a general theory of (bipartite) Euler systems. 
The axioms (Definition \ref{es}) are designed to include
the family of cohomology classes used by Bertolini and Darmon \cite{BD03}.
The methods used to bound the associated Selmer group and to develop a
criterion for equality (Theorems \ref{esb} and \ref{rigidity}) originated
with Kolyvagin, and we follow closely the approach to Kolyvagin's
theory described by Mazur and Rubin \cite{mazur-rubin}.

Let $R$ be a principal Artinian local ring with maximal ideal $\gm$ and
residue characteristic $p>3$. 
Let $T$ be a free $R$-module of rank two 
equipped with a continuous (for the discrete topology) action of 
$G_K\define \Gal(K^\alg/K)$ for some number field $K$. 
We assume that $T$ admits a perfect, $G_K$-equivariant, alternating 
$R(1)$-valued pairing. Let 
$$
\loc_w:H^1(K,T)\map{}H^1(K_w,T)
$$
denote the localization map (we assume that we are given a fixed
embedding $K^\alg\hookrightarrow K_w^\alg$ for every place $w$).
Throughout \S \ref{S:Euler Systems} we assume that we are given
a fixed self-dual Selmer structure $(\sel,\Sigma_\sel)$ on $T$, as
defined in \S \ref{selmer modules}.

If $B$ is any $R$-module and $b\in B$ we define $\ind(b,B)$,
the \emph{index of divisibility} of $b$ in $B$, to be the largest $k\le\infty$
such that $b\in\gm^k B$.

%%%%%%%%%%%%%%%%%%%%%%%%%%%%%%%%%%%%%%%%%%%%%%%%%%%%%%%%%%%%%%

\subsection{Selmer modules}
\label{selmer modules}

%%%%%%%%%%%%%%%%%%%%%%%%%%%%%%%%%%%%%%%%%%%%%%%%%%%%%%%%%%%%%%

\begin{Def}
A \emph{Selmer structure} $(\sel,\Sigma_\sel)$ on  $T$ is
a finite set of places $\Sigma_\sel$ of $K$
containing the archimedean places, the primes at which $T$ is
ramified, and the prime $p$; and, for every place $w$ of $K$,
a choice of submodule 
$$
H^1_\sel(K_w,T)\subset H^1(K_w,T)
$$
such that $H^1_\sel(K_w,T)=H^1_\unr(K_w,T)$ for all $w\not\in\Sigma_\sel$.
Define the \emph{Selmer module} $\Sel_\sel=\Sel_\sel(K,T)$ 
associated to $\sel$ by the exactness of
$$
0\map{}\Sel_\sel\map{}H^1(K,T)\map{\oplus\loc_w}
\bigoplus_w H^1(K_w,T)/H^1_\sel(K_w,T)
$$
where the sum is over all places $w$ of $K$.
A Selmer structure $\sel$ is \emph{self-dual} if the submodule 
$H^1_\sel(K_w,T)$ is maximal isotropic under the 
(symmetric) local Tate pairing
$$
H^1(K_w,T)\times H^1(K_w,T)\map{\cup}H^2(K_w, R(1)) \iso R
$$
for every finite place $w\in\Sigma_\sel$.
\end{Def}

\begin{Rem}
Note that $p\not=2$ implies $H^1(K_w,T)=0$ for $w$ archimedean.
By Tate local duality, $H^1_\sel(K_w,T)=H^1_\unr(K_w,T)$ is maximal isotropic
for all $w\not\in \Sigma_\sel$.
\end{Rem}

\begin{Rem}\label{propagation}
If $S$ is a submodule (resp. quotient) of $T$ and $(\sel,\Sigma_\sel)$
is a Selmer structure on $T$, then there is an induced Selmer structure,
still denoted $(\sel,\Sigma_\sel)$, on $S$ defined as the
preimage of $H^1_\sel(K_w,T)$ under 
$H^1(K_w,S)\map{} H^1(K_w,T)$
(resp. the image of $H^1_\sel(K_w,T)$ under 
$H^1(K_w,T)\map{} H^1(K_w,S))$
for every place $w$ of $K$. By \cite[Lemma 1.1.9]{mazur-rubin},
$H^1_\sel(K_w,S)=H^1_\unr(K_w,S)$ for every $w\not\in\Sigma_\sel$, and so 
this is well-defined. We refer to this as \emph{propagation} 
of Selmer structures.
\end{Rem}

%%%%%%%%%%%%%%%%%%%%%%%%%%%%%%%%%%%%%%%%%%%%%%%%%%%%%%%%%%%%%

\subsection{Modified Selmer modules}
\label{ordinary selmer}

%%%%%%%%%%%%%%%%%%%%%%%%%%%%%%%%%%%%%%%%%%%%%%%%%%%%%%%%%%%%

Now suppose we have a set of primes $\cl$ of $K$ which is 
disjoint from $\Sigma_\sel$ and satisfies
\begin{enumerate}
\item $\forall \gl\in\cl,\N(\gl)\not\equiv 1\pmod{p}$, 
\item $\forall \gl\in\cl$, the Frobenius $\Frob_\gl$ 
acts on $T$ with eigenvalues $\N(\gl)$ and $1$.
\end{enumerate}
Let $\cn$ denote the set of squarefree 
products of primes in $\cl$.
The two conditions above imply that $T\iso R\oplus R(1)$ as a 
$\Gal(K_\gl^\alg/K_\gl)$-module, and that the decomposition is unique.
For each $\gl\in\cl$ we define the \emph{ordinary} cohomology
$H^1_\ord(K_\gl,T)$ to be the image of 
$H^1(K_\gl,R(1))\map{}H^1(K_\gl,T).$

\begin{Lem}\label{local freeness}
For $\gl\in\cl$, the decomposition $T\iso R\oplus R(1)$ induces 
a decomposition
$$H^1(K_\gl,T)\iso H^1_\unr(K_\gl,T)\oplus H^1_\ord(K_\gl,T)$$
in which each summand is free of rank one over $R$ and is maximal isotropic
under the local Tate pairing.
\end{Lem}

\begin{proof}
By \cite[Lemma 1.3.2]{rubin}, evaluation of cocycles at $\Frob_\gl$
induces an isomorphism
$$
H^1_\unr(K_\gl,T)\iso T/(\Frob_\gl-1)T \iso R.
$$
By local class field theory
$$
H^1(K_\gl,R)\iso \Hom(\Gal(K_\gl^\unr/K_\gl),R)\iso R,
$$
again by evaluation at $\Frob_\gl$.  Thus $H^1_\unr(K_\gl,T)$
is exactly the image of $H^1(K_\gl,R)$, and is free of rank one.
Since $\N(\gl)\not\equiv 1\pmod{p}$, the pro-$p$-completion of 
$K_\gl^\times$ is canonically isomorphic to $\Z_p$, and  so
$$
H^1_\ord(K_\gl,T)\iso H^1(K_\gl, R(1))\iso R
$$ by local Kummer theory.  
The submodules $R$ and $R(1)$ of $T$ are each maximal isotropic under the
pairing $T\times T\map{}R(1)$, and so the same is true of the
spaces $H^1_\unr(K_\gl,T)$ and $H^1_\ord(K_\gl,T)$ under the cup product.
\end{proof}

\begin{Def}\label{cartesian def}
A Selmer structure $(\sel,\Sigma_\sel)$ is \emph{cartesian}
if for every quotient $T/\gm^i T$ of $T$, every place $w\in\Sigma_\sel$,
and any generator $\pi\in\gm$, the isomorphism
$$
T/\gm^i T\map{\pi^{\len(R)-i}}T[\gm^i]
$$
induces an isomorphism
$
H^1_{\sel}(K_w,T/\gm^i) \iso H^1_{\sel}(K_w,T[\gm^i]).
$
\end{Def}

\begin{Rem}
A Selmer structure $(\sel,\Sigma_\sel)$ is cartesian if and only if 
it defines a cartesian local condition, for every $w\in\Sigma_\sel$,
on the quotient category
$\mathrm{Quot}(T)$ in the sense of \cite[Definition 1.1.4]{mazur-rubin}.
\end{Rem}

\begin{Hyp}\label{cartesian}
For the remainder of \S \ref{S:Euler Systems} 
we make the following assumptions
\begin{enumerate}
\item the residual representation $T/\gm T$ is absolutely irreducible,
\item $\sel$ is cartesian.
\end{enumerate}
\end{Hyp}

\begin{Def}
For any $\mathfrak{abc}\in\cn$ we define a Selmer structure
$(\sel^\ga_\gb(\mathfrak{c}), \Sigma_{\sel^\ga_\gb(\mathfrak{c})})$
as follows: $\Sigma_{\sel^\ga_\gb(\mathfrak{c})}$ is $\Sigma_\sel$
together with all prime divisors of $\mathfrak{abc}$, 
$$
H^1_{\sel^\ga_\gb(\mathfrak{c})}(K_w,T)=H^1_\sel(K_w,T)
$$
for $w$ prime to $\mathfrak{abc}$, and 
$$
H^1_{\sel^\ga_\gb(\mathfrak{c})}(K_\gl,T)
=\left\{\begin{array}{ll}
H^1(K_\gl,T) & \mathrm{if\ }\gl|\ga\\
0 & \mathrm{if\ }\gl|\gb\\
H^1_\ord(K_\gl,T) & \mathrm{if\ }\gl|\mathfrak{c}.
\end{array}\right.
$$
If any one of $\ga$, $\gb$, $\mathfrak{c}$ is the empty product
we omit it from the notation.
\end{Def}

\begin{Lem}\label{subquotients}
The Selmer structure $\sel(\gn)$ is cartesian for any $\gn\in\cn$.
For any choice of generator $\pi\in\gm$ and any $0\le i\le\len(R)$, the 
composition
$$
T/\gm^i T\map{\pi^{\len(R)-i}}T[\gm^i]\map{}T
$$
induces isomorphisms
$$
\Sel_{\sel(\gn)}(K,T/\gm^i) \iso \Sel_{\sel(\gn)}(K,T[\gm^i])
\iso \Sel_{\sel(\gn)}(K,T)[\gm^i].
$$
\end{Lem}

\begin{proof}
For any prime $w$ not dividing $\gn$,
$H^1_{\sel(\gn)}(K_w,T)$
satisfies the cartesian property by Hypothesis \ref{cartesian}.  
For $\gl|\gn$, the cartesian property follows from the 
canonical isomorphism $H^1_\ord(K_\gl,T)\iso R$
used in the proof of Lemma \ref{local freeness}.  The second claim
now follows as in \cite[Lemma 3.5.4]{mazur-rubin}.
\end{proof}

\begin{Prop}\label{structure}
For any $\gn\in\cn$ there is a (non-canonical) decomposition
$$
\Sel_{\sel(\gn)}\iso R^{e(\gn)}\oplus M_\gn\oplus M_\gn
$$
with $e(\gn)\in\{0,1\}$.
\end{Prop}

\begin{proof}
This follows from the existence of a modified form of the
Cassels-Tate pairing, together with the self-duality hypotheses
on $T$ and $\sel$; see \cite[Theorem 1.4.2]{me}.
\end{proof}

\begin{Def}\label{stub}
Let $\cn^\even\subset\cn$ be the subset for which $e(\gn)=0$,
and $\cn^\odd\subset\cn$ the subset for which $e(\gn)=1$.
For $\gn\in\cn$ we define the \emph{stub module}
$$
\Stub_{\gn}=\left\{\begin{array}{ll}
\gm^{\len(M_\gn)}\cdot R & \mathrm{if\ }\gn\in\cn^\even\\
\gm^{\len(M_\gn)}\cdot \Sel_{\sel(\gn)}(K,T)
& \mathrm{if\ }\gn\in\cn^\odd \end{array}\right.
$$
with $M_\gn$ as in Proposition \ref{structure}. 
Note that $\Stub_\gn$ is a cyclic $R$-module for every $\gn\in\cn$. 
\end{Def}

The following proposition is a consequence of Poitou-Tate
global duality, and is similar to \cite[Lemma 1.5.8]{me} and 
\cite[Lemma 4.1.6]{mazur-rubin}. Our self-duality assumptions, 
together with the fact that the local conditions $H^1_\unr(K_\gl,T)$ and
$H^1_\ord(K_\gl,T)$ have rank one, give a much stronger result.

\begin{Prop}\label{global duality}
For any $\gn\gl\in\cn$ there are non-negative integers $a,b$ with
$a+b=\len(R)$ such that in the diagram of inclusions
$$
\xymatrix{
 & {\Sel_{\sel^\gl(\gn)}}   \\
{\Sel_{\sel(\gn)}}\ar[ur]^b  &  &  {\Sel_{\sel(\gn\gl)}}\ar[ul]_a  \\
 & {\Sel_{\sel_\gl(\gn)}\ar[ul]^a\ar[ur]_b  }
}
$$
the labels on the arrows are the lengths of the respective 
quotients.  All four quotients are cyclic $R$-modules
and
\begin{equation}\label{pretty diagram}
a=\len(\loc_\gl(\Sel_{\sel(\gn)}))
\hspace{1cm}
b=\len(\loc_\gl(\Sel_{\sel(\gn\gl)})).
\end{equation}
\end{Prop}

\begin{proof}
Take (\ref{pretty diagram}) as the definition of $a$ and $b$,  so that
$a$ and $b$ are the lengths of the lower left and right quotients, 
respectively. The cyclicity of the quotients 
follows from Lemma \ref{local freeness};
for example the lower left quotient injects into $H^1_\unr(K_\gl,T)$.

Exactly as in the proof of \cite[Lemma 1.5.7]{me}, the quotient
\begin{equation}\label{global quotient}
\Sel_{\sel^\gl(\gn)}/\big(\Sel_{\sel(\gn)}+\Sel_{\sel(\gn\gl)}\big)
\end{equation}
admits a nondegenerate, alternating $R$-valued pairing.  
The pairing is defined as
follows: given $x,y\in\Sel_{\sel^\gl(\gn)}$, let $x'$ be the projection
of $\loc_\gl(x)$ to $H^1_\unr(K_\gl,T)$, and let $y'$ be the
projection of $\loc_\gl(y)$ to $H^1_\ord(K_{\gl},T)$.  The pairing
of $x$ and $y$ is then defined to be the local Tate pairing of $x'$ and $y'$.
The quotient  (\ref{global quotient}) is a 
cyclic $R$-module, and so the existence of such a pairing implies that
it is trivial.  

Directly from the definitions we have 
$$
\Sel_{\sel_\gl(\gn)}= \Sel_{\sel(\gn)}\cap \Sel_{\sel(\gl\gn)}.
$$
Combining this with the above,  it follows that the 
lower left quotient is isomorphic to the  upper right,
and the lower right is isomorphic to the upper left.  This
proves everything except for the claim $a+b=\len(R)$, which is a consequence
of global duality as in \cite[Lemma 1.5.8]{me} or 
\cite[Lemma 4.1.6]{mazur-rubin}.
\end{proof}

\begin{Cor}\label{rho}
Fix $\gn\in\cn$ and let $e(\gn)$ be as in Proposition \ref{structure}. 
The integer
$$
\rho(\gn)=
\mathrm{dim}_{R/\gm}\big(\Sel_{\sel(\gn)}(K,T/\gm T)\big)
$$
satisfies $e(\gn)\equiv\rho(\gn)\pmod{2}$, and for any $\gl\in\cl$ prime
to $\gn$
\begin{eqnarray*}
\rho(\gn\gl)=\rho(\gn)+1 &\iff&
\loc_\gl(\Sel_{\sel(\gn)}(K,T/\gm T)=0\\
\rho(\gn\gl)=\rho(\gn)-1 &\iff&
\loc_\gl(\Sel_{\sel(\gn)}(K,T/\gm T)\not=0.
\end{eqnarray*}
The claim continues to hold, and the value of $\rho(\gn)$ 
remains unchanged, if one replaces $\Sel_{\sel(\gn)}(K,T/\gm T)$ by 
$\Sel_{\sel(\gn)}(K,T)[\gm]$ everywhere.
\end{Cor}

\begin{proof}
Apply Proposition \ref{global duality} with $T$ replaced by $T/\gm T$.
Then $$\rho(\gn\gl)=\rho(\gn)-a+b,$$ $a+b=1$, and $a=0$ if and only if
$\loc_\gl$ kills $\Sel_{\sel(\gn)}(K,T/\gm T)$.  Combining this 
with Lemma \ref{subquotients} proves the claim.
\end{proof}

\begin{Rem}
Note that Corollary \ref{rho} implies that for $\gn\gl\in\cn$,
$$
\gn\in\cn^\even\iff \gn\gl\in\cn^\odd.
$$
We will use this repeatedly throughout.
\end{Rem}

\begin{Cor}\label{length induction}
Suppose $\gn\gl\in\cn$, and let $a$ and $b$ be as in Proposition 
\ref{global duality}.  Then
$$
\len(M_\gn)=\left\{\begin{array}{ll}
\len(M_{\gn\gl})+a&\mathrm{\ if\ } \gn\in\cn^\even\\
\len(M_{\gn\gl})-b&\mathrm{\ if\ } \gn\in\cn^\odd.
\end{array}\right.
$$
\end{Cor}

\begin{proof}
Suppose $\gn\in\cn^\even$.  Then
\begin{eqnarray*}
2\cdot \len(M_\gn) &=& \len(\Sel_{\sel(\gn)})  \\
&=&  \len(\Sel_{\sel(\gn\gl)})-b+a \\
&=&  \len(\Sel_{\sel(\gn\gl)})+2a-\len(R) \\
&=&  2\cdot \len(M_{\gn\gl}) +2a.
\end{eqnarray*}
The case $\gn\in\cn^\odd$ is similar.
\end{proof}

\begin{Cor}\label{stub shift}
Suppose $\gn\gl\in\cn$. 
There is an isomorphism
of $R$-modules
\begin{eqnarray*}
\loc_\gl(\Stub_\gn) \iso \Stub_{\gn\gl}
& &\mathrm{if\ }\gn\in\cn^\odd \\
\loc_\gl(\Stub_{\gn\gl}) \iso \Stub_\gn
& &\mathrm{if\ }\gn\in\cn^\even.
\end{eqnarray*}
\end{Cor}

\begin{proof}
Suppose $\gn\in\cn^\odd$. Since the modules in question are cyclic, it
suffices to check that they have the same 
annihilator in $R$.  The image of $\Stub_\gn$ in $H^1_\unr(K_\gl,T)$
is annihilated by $\gm^i$ if and only if $\gm^{i+\len(M_\gn)}$
kills the lower left quotient in the diagram of Proposition 
\ref{global duality}, that is, if and only if
$a\le i+\len(M_\gn)$.  By Corollary \ref{length induction}, 
this is equivalent to $\len(R)\le i+ \len(M_{\gn\gl})$, which is 
equivalent to $\gm^{i}\cdot \Stub_{\gn\gl}=0$.  
The case $\gn\in\cn^\even$ is similar.
\end{proof}

%%%%%%%%%%%%%%%%%%%%%%%%%%%%%%%%%%%%%%%%%%%%%%%%%%%%%%%%%%%%%

\subsection{Euler systems}
\label{esb subsection}

%%%%%%%%%%%%%%%%%%%%%%%%%%%%%%%%%%%%%%%%%%%%%%%%%%%%%%%%%%%%

Continue to assume that Hypothesis \ref{cartesian} holds, as well as

\begin{Hyp}\label{useful primes}
For any $c\in H^1(K,T/\gm T)$ there are infinitely many $\gl\in\cl$ such that
$\loc_\gl(c)\not=0$.
\end{Hyp}

\begin{Def}\label{es}
A \emph{bipartite Euler system of odd type} for $(T,\sel,\cl)$ 
is a pair of families
$$
\{ \esodd_\gn\in\Sel_{\sel(\gn)}(K,T) \mid \gn\in\cn^\odd \}
\hspace{1cm}
\{ \eseven_\gn\in R \mid \gn\in\cn^\even \}
$$
related by the first and second reciprocity laws:
\begin{enumerate}
\item for any $\gn\gl\in\cn^\odd$, there exists an isomorphism of 
$R$-modules  
$$
R/(\eseven_\gn)\iso
H^1_\ord(K_{\gl},T)/R\cdot\loc_\gl(\esodd_{\gn\gl}),
$$
\item for any $\gn\gl\in\cn^\even$, there exists an isomorphism of 
$R$-modules  
$$
R/(\eseven_{\gn\gl})\iso
H^1_\unr(K_{\gl},T)/R\cdot\loc_\gl(\esodd_{\gn}).
$$
\end{enumerate}
A \emph{bipartite Euler system of even type} is defined in the same way,
but with even and odd interchanged everywhere in the definition.
\end{Def}

From here on we drop the adjective ``bipartite'', and simply call such a 
pair of families an Euler system.
An Euler system (of even or odd type) is \emph{nontrivial} if 
$\eseven_\gn\not=0$
for some $\gn$ (of the appropriate type).  By the reciprocity laws
and the following lemma, this is equivalent to $\esodd_\gn\not=0$
for some $\gn$ of the appropriate type.

\begin{Lem}\label{local surjectivity}
For any $\gn\in\cn$ and any cyclic $R$-submodule 
$C\subset \Sel_{\sel(\gn)}$, there are
infinitely many $\gl\in\cl$ such that $\loc_\gl$ takes $C$
injectively into $H^1_\unr(K_\gl,T)$. If $C$ is free of rank one then
for such any such  $\gl$, $\loc_\gl$ takes $C$ isomorphically onto
$H^1_\unr(K_\gl,T)$.
\end{Lem}

\begin{proof}
Let $i$ be maximal such that $\gm^i C\not=0$, so that 
$\gm^iC\subset\Sel_{\sel(\gn)}[\gm]$.
By Hypothesis \ref{useful primes} and Lemma \ref{subquotients}, 
there are infinitely many $\gl\in\cl$ prime to 
$\gn$ such that $\loc_\gl(\gm^iC)\not=0$.  For any such prime $\loc_\gl$
takes $C$ injectively into $H^1_\unr(K_\gl,T)$. 
The final claim is immediate from 
Lemma \ref{local freeness}.
\end{proof}

\begin{Prop}\label{no even es}
There are no nontrivial Euler systems for $(T,\sel,\cl)$ of even type.
\end{Prop}

\begin{proof}
Given a nontrivial Euler system of even type we fix $\gn\in \cn^\odd$ 
such that $\eseven_\gn\not=0$.  In the notation of Proposition \ref{structure},
$e(\gn)=1$, and so $\Sel_{\sel(\gn)}(K,T)$
contains a free rank-one $R$-submodule $C$.  
If $\gl\in\cl$ is chosen as in
Lemma \ref{local surjectivity} then the natural injection
$$
\Sel_{\sel(\gn)}/\Sel_{\sel_\gl(\gn)}\hookrightarrow H^1_\unr(K_\gl,T)\iso R
$$
is an isomorphism, and Proposition \ref{global duality} implies that
$\Sel_{\sel_\gl(\gn)}=\Sel_{\sel(\gn\gl)}$.  In particular
$\loc_\gl(\esodd_{\gn\gl})=0$, violating the first reciprocity law.
\end{proof}

\begin{Prop}\label{annihilation}
Fix an Euler system of odd type for $(T,\sel,\cl)$, let $k$
be the length of $R$, and let
$M_\gn$ be as in Proposition \ref{structure}.  If $\eseven_\gn\not=0$
for some $\gn\in\cn^\even$ then $M_\gn$ is killed by $\gm^{k-1}$.
If $\esodd_\gn\not=0$
for some $\gn\in\cn^\odd$ then $M_\gn$ is killed by $\gm^{k-1}$.
\end{Prop}

\begin{proof}
First suppose $\gn\in\cn^\even$ and $\gm^{k-1}M_\gn\not=0$.  Then
$\Sel_{\sel(\gn)}$ contains a free rank one submodule, $C$.
Choose $\gl\in\cl$ not dividing $\gn$ such that 
$\loc_\gl$ takes $C$ isomorphically onto $H^1_\unr(K_\gl,T)$ 
(Lemma \ref{local surjectivity}).  By Proposition \ref{global duality},
$\loc_\gl(\Sel_{\sel(\gn\gl)})=0$.  Thus $\loc_\gl(\esodd_{\gn\gl})=0$
and the first reciprocity law implies that $\eseven_\gn=0$.

Now suppose $\gn\in\cn^\odd$ and $\gm^{k-1}M_\gn\not=0$.  Proposition
\ref{structure} implies that $\Sel_{\sel(\gn)}$ contains a free
submodule of rank two, $C$.  From this and Lemma 
\ref{local freeness} one may deduce that
for any $\gl\in\cl$ the kernel of 
$$
\loc_\gl:\Sel_{\sel(\gn)}\map{}H^1_\unr(K_\gl,T)
$$
contains a free submodule.  This kernel is exactly 
$\Sel_{\sel_\gl(\gn)}\subset \Sel_{\sel(\gn\gl)}$,
and so $\gm^{k-1}M_{\gn\gl}\not=0$.
By the case considered above $\eseven_{\gn\gl}=0$, 
and the second reciprocity law
implies that $\loc_\gl(\esodd_\gn)=0$.
Since this holds for all choices of $\gl$, $\esodd_\gn=0$ by
Lemma \ref{local surjectivity}.
\end{proof}

The above proposition shows that an Euler system
gives a (somewhat weak) annihilation result for Selmer groups.  To strengthen
this to an upper bound on Selmer groups, we must impose the 
hypothesis of \emph{freeness} defined below.  For an example of
how this hypothesis may be verified in practice, see the proof
of Lemma \ref{free es}.

\begin{Def}\label{free}
We will say that an Euler system of odd type is \emph{free}
if for every $\gn\in\cn^\odd$, there is a free rank-one 
$R$-submodule $C_\gn\subset \Sel_{\sel(\gn)}$ containing $\esodd_\gn$.
\end{Def}

\begin{Thm}\label{esb}
For any free Euler system of odd type for $(T,\sel,\cl)$,
$\eseven_\gn\in\Stub_\gn$ for every  $\gn\in\cn^\even$, and 
$\esodd_\gn\in\Stub_\gn$ for every $\gn\in\cn^\odd$.
Equivalently, the $R$-module $M_\gn$ of Proposition 
\ref{structure} satisfies
$$
\len(M_\gn)\le \left\{\begin{array}{ll}
\ind(\eseven_\gn, R) & \mathrm{if\ }\gn\in\cn^\even \\
\ind(\esodd_\gn, \Sel_{\sel(\gn)}(K,T)) & \mathrm{if\ }\gn\in\cn^\odd.
\end{array}\right.
$$
\end{Thm}

\begin{proof}
The proof is by induction on 
$\rho(\gn)$, as defined in Corollary \ref{rho}.
If $\rho(\gn)=0$ then $M_n=0$, and so
$\gn\in\cn^\even$, $\Stub_\gn=R$, and the claim is vacuous.
Similarly, if $\rho(\gn)=1$ then $\gn\in\cn^\odd$, $M_n=0$, and the 
claim is vacuous.  We assume now that $\rho(\gn)\ge 2$, so that $M_n\not=0$.

Suppose $\gn\in\cn^\even$ and $\eseven_\gn\not=0$.
Fix any $\gl\in\cl$ prime to $\gn$ such that the Selmer group
$\Sel_{\sel(\gn)}(K,T/\gm T)$ is not killed by $\loc_\gl$.
By Corollary \ref{rho} and the induction hypothesis, 
$\esodd_{\gn\gl}\in\Stub_{\gn\gl}$, and so Corollary \ref{length induction}
gives
\begin{eqnarray*}
\len(M_n) & = & \len(M_{\gn\gl})+a\\
&\le &\ind(\esodd_{\gn\gl},\Sel_{\sel(\gn\gl)})+a\\
&\le &\ind\big(\loc_\gl(\esodd_{\gn\gl}),
\loc_\gl(\Sel_{\sel(\gn\gl)})\big)+a
\end{eqnarray*}
where $a$ is as in Proposition \ref{global duality}.  
The first reciprocity law implies
\begin{eqnarray*}
\ind(\eseven_\gn,R) &=& \ind(\loc_\gl(\esodd_{\gn\gl}),H^1_\ord(K_\gl,T))\\
&=& \ind(\loc_\gl(\esodd_{\gn\gl}),\loc_\gl(\Sel_{\sel(\gn\gl)})\big)
+\len(H^1_\ord(K_\gl,T)/\loc_\gl(\Sel_{\sel(\gn\gl)}))\\
&=& \ind(\loc_\gl(\esodd_{\gn\gl}),\loc_\gl(\Sel_{\sel(\gn\gl)})\big)
+\len(R)-b.
\end{eqnarray*}
Since $a+b=\len(R)$, we conclude $\len(M_n)\le \ind(\eseven_\gn,R)$.

Now suppose $\gn\in\cn^\odd$ and $\esodd_\gn\not=0$.  
Let $C_\gn$ be as in Definition \ref{free}
and fix $\gl\in\cl$ prime to $\gn$ such that $\loc_\gl$ takes $C_\gn$
isomorphically onto $H^1_\unr(K_\gl,T)$ (using Lemma \ref{local surjectivity}).
Again applying Corollary \ref{rho}, 
$\rho(\gn\gl)=\rho(\gn)-1$, and so $\eseven_{\gn\gl}\in\Stub_{\gn\gl}$.
By Corollary \ref{length induction} (with $a=\len(R)$ and $b=0$)
and the second reciprocity law,
\begin{eqnarray*}
\len(M_n) = \len(M_{\gn\gl}) &\le &\ind(\eseven_{\gn\gl},R)\\
&= &\ind\big(\loc_\gl(\esodd_{\gn}), H^1_\unr(K_\gl,T)\big)\\
&=&\ind\big(\esodd_{\gn},\Sel_{\sel(\gn)}).
\end{eqnarray*}
\end{proof}

%%%%%%%%%%%%%%%%%%%%%%%%%%%%%%%%%%%%%%%%%%%%%%%%%%%%%%%%%%%%%%%%%

\subsection{Sheaves on graphs}
\label{sheaf section}

%%%%%%%%%%%%%%%%%%%%%%%%%%%%%%%%%%%%%%%%%%%%%%%%%%%%%%%%%%%%%%%%

Let $\xgraph$ be the graph whose vertices $v(\gn)$ are indexed by $\gn\in\cn$,
and an edge $e(\gn,\gn\gl)$ connects $v(\gn)$ to $v(\gn\gl)$ whenever
$\gl\in\cl$ and $\gn\gl\in\cn$.  
A vertex $v(\gn)$ will be called even or odd, depending 
on whether $\gn$ lies in $\cn^\even$ or $\cn^\odd$, and every edge 
connects an even vertex to an odd one (by Corollary \ref{rho}).

We define a sheaf $\mathrm{ES}(\xgraph)$ on $\xgraph$ in the sense of 
\cite[\S 3.1]{mazur-rubin} called the \emph{Euler system sheaf}
as follows.
To each vertex $v=v(\gn)$ we attach the 
$R$-module 
$$
\mathrm{ES}(v)=\left\{\begin{array}{ll}
\Sel_{\sel(\gn)} & \mathrm{if\ }\gn\in\cn^\odd \\
R & \mathrm{if\ }\gn\in\cn^\even\end{array}\right.
$$
and to each edge $e=e(\gn,\gn\gl)$
we attach the $R$-module
$$
\mathrm{ES}(e)=\left\{\begin{array}{ll}
H^1_\unr(K_\gl,T) &\mathrm{if\ }\gn\in\cn^\odd \\
H^1_\ord(K_\gl,T ) &\mathrm{if\ }\gn\in\cn^\even.
\end{array}\right.
$$
If $e=e(\gn,\gn\gl)$ is an edge with endpoint $v$, we define the 
\emph{vertex-to-edge map} 
$$\psi^e_v:\mathrm{ES}(v)\map{}\mathrm{ES}(e)$$ as follows.
If $v$ is odd then
$$
\psi^e_v=\loc_\gl:
\left\{\begin{array}{ll}
\Sel_{\sel(\gn)}\map{}H^1_\unr(K_\gl,T)
& \mathrm{if\ }v=v(\gn)\\
\Sel_{\sel(\gn\gl)}\map{}H^1_\ord(K_\gl,T)
& \mathrm{if\ }v=v(\gn\gl).
\end{array}\right.
$$
If $v$ is even then fix, using Lemma 
\ref{local freeness}, an isomorphism
\begin{equation}\label{edge choice}
\psi^e_v: R\iso \left\{\begin{array}{ll}
H^1_\unr(K_\gl,T) & \mathrm{if\ } v=v(\gn\gl)\\
H^1_\ord(K_\gl,T) & \mathrm{if\ } v=v(\gn).
\end{array}\right.
\end{equation}
Of course, the choice of isomorphism (\ref{edge choice}) is not unique,
but we fix a choice, for each edge $e$ with even vertex $v$, once and for all.

\begin{Def}\label{stub def}
The Euler system sheaf has a locally cyclic (in the sense
of \cite[Definition 3.4.2]{mazur-rubin})
subsheaf, the \emph{stub sheaf} $\Stub(\xgraph)$, defined as follows. 
To each vertex $v=v(\gn)$ we attach the 
cyclic $R$-module 
$$
\Stub(v)=\Stub_\gn\subset \mathrm{ES}(v),
$$ 
and to each edge $e=e(\gn,\gn\gl)$
we attach the cyclic module $\Stub(e)\subset \mathrm{ES}(e)$ 
$$
\Stub(e)=\left\{\begin{array}{ll}
\loc_\gl(\Stub_\gn) &\mathrm{if\ }\gn\in\cn^\odd \\
 \loc_\gl(\Stub_{\gn\gl}) &\mathrm{if\ }\gn\in\cn^\even.
\end{array}\right.
$$
If $e$ is an edge connecting the vertices $v$ and $v'$, with $v$ even and
$v'$ odd, then the vertex-to-edge map $\psi_{v'}^e$ 
restricts to a surjective map $\Stub(v')\map{}\Stub(e)$.
By Corollary \ref{stub shift},
the map $\psi_{v}^e$ restricts to an isomorphism
$\Stub(v)\iso \Stub(e)$.
\end{Def}

\begin{Def}
A \emph{core vertex} of $\xgraph$ is a vertex $v$ such that $\Stub(v)\iso R$.
\end{Def}

\begin{Rem}\label{core remark}
Set $\tbar=T/\gm T$. For $\gn\in\cn$, recall the integer
$$\rho(\gn)=\dim_{R/\gm}(\Sel_{\sel(\gn)}(K,\tbar))$$ of
Corollary \ref{rho}. 
It is clear from Lemma \ref{subquotients} and 
Proposition \ref{structure} that $v(\gn)$
is a core vertex if and only if $\rho(\gn)=0$ or $1$.  
\end{Rem}

The \emph{core subgraph} $\xgraph_0\subset\xgraph$ is the graph whose
vertices are the core vertices of $\xgraph$, with two vertices connected
by an edge in $\xgraph_0$ if and only if they are connected by an edge in
$\xgraph$.  We let  $\Stub(\xgraph_0)$ be the restriction of $\Stub(\xgraph)$
to $\xgraph_0$ in the obvious sense.  

\begin{Lem}\label{free sheaves}
The sheaf $\Stub(\xgraph_0)$ is 
locally free of rank one.  That is to say,
\begin{itemize}
\item for every vertex $v$ of $\xgraph_0$, $\Stub(v)$ is free of rank one,
\item for every edge $e$ of $\xgraph_0$, $\Stub(e)$ is free of rank one,
\item for every edge $e$ of $\xgraph_0$
with endpoint $v$, the vertex-to-edge map
$$\psi_v^e:\Stub(v)\map{}\Stub(e)$$ is an isomorphism.
\end{itemize}
\end{Lem}

\begin{proof}
The first property is the definition of $\xgraph_0$.
As noted in Definition \ref{stub def}, $\Stub(e)$ is isomorphic
to $\Stub(v)$, where $v$ is the even endpoint of $e$.
This proves the second property.
The final property follows from the first two, together with the surjectivity
of the vertex-to-edge maps in $\Stub(\xgraph)$.
\end{proof}

\begin{Def}
A \emph{global section}, $s$, of the sheaf $\Stub(\xgraph)$ on $\xgraph$ 
is a function on vertices and edges of $\xgraph$, 
$$
v\mapsto s(v)\in \Stub(v)
\hspace{1cm}
e\mapsto s(e)\in\Stub(e),
$$
such that for every edge $e$ with endpoints $v$ and $v'$
$$
\psi^e_v(s(v))=s(e)=\psi^e_{v'}(s(v'))
$$
in $\Stub(e)$. A global section of $\mathrm{ES}(\xgraph)$ 
is defined in the same  way.
\end{Def}

\begin{Def}
For two vertices $v$ and $v'$ of $\xgraph$, a \emph{path} from $v$ to $v'$
in $\xgraph$ is a finite  sequence of vertices $v=v_0, v_1,\ldots, v_k=v'$ 
such that $v_i$ is connected to $v_{i+1}$ by an edge $e_i$.  
A path is \emph{surjective} (for the locally cyclic sheaf $\Stub(\xgraph)$) 
if the vertex-to-edge map 
$$
\psi_{v_{i+1}}^{e_i}:\Stub(v_{i+1})\map{}\Stub(e_i)
$$ 
is an isomorphism for every $i$. We make the same definitions for $\xgraph_0$.
\end{Def}

\begin{Rem}\label{surjective remark}
Note that a surjective path from 
$v$ to $v'$ induces in an obvious way (\cite[\S 3.4]{mazur-rubin})
a surjective map $\Stub(v)\map{}\Stub(v')$, 
and that for any global section $s$ of $\Stub(\xgraph)$
this map takes $s(v)$ to $s(v')$.
\end{Rem}

\begin{Lem}\label{surjective paths}
A path $v_0, \ldots, v_k$ in $\xgraph$ is surjective if and only if
$$\len(\Stub(v_{i+1})) \le \len(\Stub(v_i))$$
for every $i$.
\end{Lem}

\begin{proof}
Suppose we are given a path in $\xgraph$
from $v_0$ to $v_k$. 
If $v_i$ is odd and $v_{i+1}$ is even then the vertex-to-edge map
$\Stub(v_{i+1})\map{}\Stub(e_i)$ is an isomorphism, 
while $\Stub(v_i)\map{} \Stub(e_i)$
is surjective.  Thus the path $v_i, v_{i+1}$ is surjective and
$\len(\Stub(v_{i+1})) \le \len(\Stub(v_i))$.

If $v_i$ is even and $v_{i+1}$ is odd
then the vertex-to-edge map 
$\Stub(v_{i+1})\map{}\Stub(e_i)$ is surjective, while 
$\Stub(v_i)\map{} \Stub(e_i)$
is an isomorphism.  In particular 
$$\len(\Stub(v_{i+1})) \ge \len(\Stub(v_i)).$$
Thus $\psi_{v_{i+1}}^{e_i}$ is injective if and only
if it is an isomorphism. This is equivalent to 
$\Stub(v_{i+1})\iso \Stub(v_i)$, which is equivalent to
$$\len(\Stub(v_{i+1})) \le \len(\Stub(v_i)).$$
Since $v_0,\ldots,v_k$ is surjective if and only if $v_i,v_{i+1}$
is a surjective path for every $i$, the claim follows.
\end{proof}

\begin{Lem}\label{cores}
For any vertex $v$ of $\xgraph$ there is a core vertex $v_0$ and a 
surjective path in $\xgraph$ from $v_0$ to $v$.
For any $\gn\in\cn$ there is a $\gn'\in\cn$ with $\gn|\gn'$
such that $v(\gn')$ is a core vertex, and $\gn'$ may be chosen either in 
$\cn^\even$ or in $\cn^\odd$.
\end{Lem}

\begin{proof}
Set $w_0=v$, and construct a sequence of vertices $w_i$
inductively as follows.

If $w_i=w(\gn_i)$ is even and not a core vertex, 
choose $\gl\in\cl$ prime to $\gn_i$ such that 
$\loc_\gl(\Sel_{\sel(\gn_i)})\not=0$, and set $w_{i+1}=w(\gn_i\gl)$.  
In the notation of Corollary \ref{length induction}, $a>0$, and so
$$\len(\Stub(w_i))<\len(\Stub(w_{i+1})).$$  
If $w_i$ is already an even core vertex, then a similar argument 
shows that $$\len(\Stub(w_i))=\len(\Stub(w_{i+1}))$$ for any choice of $\gl$.

If $w_i=w(\gn_i)$ is odd 
choose (using Lemma \ref{local surjectivity})
$\gl\in\cl$ prime to $\gn_i$ such that localization at $\gl$
takes a free rank-one submodule of $\Sel_{\sel(\gn_i)}$ 
isomorphically onto $H^1_\unr(K_\gl,T)$, and set $w_{i+1}=w(\gn_i\gl)$.  
In the notation of Corollary \ref{length induction}, 
$a=\len(R)$ and $b=0$, and so
$$\len(\Stub(w_i))=\len(\Stub(w_{i+1})).$$

Eventually $\len(\Stub(w_k))=\len(R)$, and we have constructed a path from 
$v$ to a core vertex $v_0=w_k$. By construction
$$
\len(\Stub(w_i))\le\len(\Stub(w_{i+1}))
$$
for every $i$, and so Lemma \ref{surjective paths} implies
that the path $w_k, w_{k-1},\ldots, w_0$
is a surjective path from $v_0$ to $v$.
The final claim is clear from the construction above.
\end{proof}

For any $\ga\in\cn$, let $\xgraph_{0,\ga}$ be the subgraph
of $\xgraph_0$ whose vertices consist of those core vertices $v(\gn)$ with 
$\ga|\gn$.  Two vertices are connected by an edge in $\xgraph_{0,\ga}$
if and only if they are connected by an edge in $\xgraph_0$.

\begin{Lem}\label{pre-connected}
If $v(\ga)$ is a core vertex then the graph $\xgraph_{0,\ga}$
is path connected.
\end{Lem}

\begin{proof}
Fix $\gn=\ga\gb\in\cn$.  We show by induction on the number of prime
factors of $\gb$ that there is a path in $\xgraph_{0,\ga}$
from $v(\gn)$ to $v(\ga)$.  Assume $\gb>1$, otherwise there is nothing
to prove.  First suppose $v(\gn)$ is even, so that $\rho(\gn)=0$.
By Corollary \ref{rho}, $\rho(\gn/\gl)=1$ for any $\gl\in\cl$ dividing $\gb$.
Hence $v(\gn/\gl)$ is a vertex in $\xgraph_{0,\ga}$ connected to $v(\gn)$
by an edge, and by the induction hypothesis there is a path in 
$\xgraph_{0,\ga}$ from $v(\gn/\gl)$ to $v(\ga)$.
Similarly, if $v(\gn)$ is odd and
$\loc_\gl(\Sel_{\sel(\gn)}(K,\tbar))\not=0$
for some $\gl$ dividing $\gb$, then Proposition \ref{global duality}
implies $\rho(\gn/\gl)=\rho(\gn)-1=0$, and again by we are done by the
induction hypothesis.

We are left to treat the case where $\gn\in\cn^\even$ and 
$\loc_\gl(\Sel_{\sel(\gn)}(K,\tbar))$
is trivial for every $\gl$ dividing $\gb$.  Thus
$$
\Sel_{\sel(\gn)}(K,\tbar)=\Sel_{\sel_\gb(\ga)}(K,\tbar)
\subset \Sel_{\sel(\ga)}(K,\tbar).
$$
Since the $R/\gm$-vector space on the left has dimension $\rho(\gn)=1$
while the space on the right has dimension $\rho(\ga)\le 1$, we conclude
that the above inclusion is an equality.  In particular
$$
\Sel_{\sel(\gn)}(K,\tbar)
=\Sel_{\sel_{\gb'}(\ga)}(K,\tbar)
=\Sel_{\sel(\ga)}(K,\tbar)
$$
for any $\gb'|\gb$.  Take $\gb'=\gb/\gq$ for some prime $\gq$ dividing
$\gb$, and let $\gl\in\cl$ be any prime not dividing 
$\gn$ such that $\loc_\gl(\Sel_{\sel(\gn)}(K,\tbar))\not=0$.
By Corollary \ref{rho}, $\rho(\ga\gb')=2$, $\rho(\gn\gl)=0$,
and, since 
$$
\Sel_{\sel(\gn)}(K,\tbar)=\Sel_{\sel_{\gb}(\ga)}(K,\tbar)
\subset \Sel_{\sel(\ga\gb')}(K,\tbar)
$$
is not killed by $\loc_\gl$, $\rho(\ga\gb'\gl)=1$.
Thus $v(\ga\gb)$, $v(\ga\gb\gl)$, $v(\ga\gb'\gl)$ 
is a path in $\xgraph_{0,\ga}$. Finally, if 
$\loc_\mathfrak{r}(\Sel_{\sel(\ga\gb'\gl)}(K,\tbar))=0$
for every $\mathfrak{r}\in\cl$ dividing $\gb'\gl$,
then 
$$
\Sel_{\sel(\ga\gb'\gl)}(K,\tbar)=
\Sel_{\sel_{\gb'\gl}(\ga)}(K,\tbar)\subset
\Sel_{\sel_{\gb'}(\ga)}(K,\tbar)= \Sel_{\sel(\gn)}(K,\tbar).
$$
The Selmer groups on the left and right are both one dimensional
over $R/\gm$, so equality holds everywhere.  This contradicts 
$\loc_\gl(\Sel_{\sel(\gn)}(K,\tbar))\not=0$, and we conclude that
$\loc_\mathfrak{r}(\Sel_{\sel(\ga\gb'\gl)}(K,\tbar))\not=0$
for some $\mathfrak{r}\in\cl$ dividing $\gb'\gl$.  This 
returns us to the case of the paragraph above, and so
$v(\gn), v(\gn\gl),v(\ga\gb'\gl), 
v(\ga\gb'\gl/\mathfrak{r})$ is a path in $\xgraph_0$.
By the induction hypothesis, this may be continued to a path 
terminating at $v(\ga)$.
\end{proof}

\begin{Prop}\label{hub}
The core subgraph is path connected and contains both even and odd
vertices. For any 
vertex $v$ of $\xgraph$ and any core vertex $v_0$ of $\xgraph$,
there is a surjective path from $v_0$ to $v$.
\end{Prop}

\begin{proof}
The fact that $\xgraph_0$ contains both even and odd vertices
follows from the final statement of Lemma \ref{cores}.
Suppose we are given two core vertices $v(\ga)$ and $v(\gb)$.
By the second part of Lemma \ref{cores} we may choose 
$\gn\in\cn$ divisible by $\ga\gb$
such that $v(\gn)$ is a core vertex.  By Lemma \ref{pre-connected},
there is a path in $\xgraph_{0,\ga}$ from $v(\ga)$ to $v(\gn)$,
and a path in $\xgraph_{0,\gb}$ from $v(\gb)$ to $v(\gn)$.  Since
any path in $\xgraph_{0,\ga}$ is also a path in $\xgraph_0$, and
similarly for $\gb$, there is a path in $\xgraph_0$ from $v(\ga)$
to $v(\gb)$.  Since any path in $\xgraph_0$ is surjective, any two
core vertices may be connected by a surjective path.  
The final claim now follows from Lemma \ref{cores}.
\end{proof}

\begin{Cor}\label{section defect}
For any global section $s$ of $\Stub(\xgraph)$ there is a unique 
$\delta=\delta(s)$ with $0\le\delta\le\len(R)$ 
such that $s(v)$ generates $\gm^\delta \Stub(v)$
for every vertex $v$ of $\xgraph$.  
The section $s$ is determined by its value at any core vertex.
\end{Cor}

\begin{proof}
Fix a core vertex $v_0$ and define $\delta$ to be such that $s(v_0)$
generates $\gm^\delta \Stub(v_0)$.  By Remark \ref{surjective remark}
and Proposition \ref{hub}, for any vertex $v$ in $\xgraph$ there is 
a surjective map $\Stub(v_0)\map{}\Stub(v)$ taking $s(v_0)$ to $s(v)$.
The claim follows.
\end{proof}

%%%%%%%%%%%%%%%%%%%%%%%%%%%%%%%%%%%%%%%%%%%%%%%%%%%%%%%%%%%%%%%%

\subsection{The rigidity theorem}

%%%%%%%%%%%%%%%%%%%%%%%%%%%%%%%%%%%%%%%%%%%%%%%%%%%%%%%%%%%%%%%%%%

\begin{Thm}\label{rigidity}
Assume Hypotheses \ref{cartesian} and \ref{useful primes}, and
suppose that we are given a nontrivial free Euler system of odd type
for $(T,\sel,\cl)$.  
There is a unique integer $\delta$, independent of $\gn\in\cn$, 
with the property that 
$\eseven_\gn$ generates $\gm^\delta \Stub_\gn$ for every $\gn\in\cn^\even$
and $\esodd_\gn$ generates $\gm^\delta\Stub_\gn$ for
every $\gn\in\cn^\odd$.  Furthermore, $\delta$ is given by
\begin{eqnarray*}
\delta  &=&  \min \{\ \ind(\eseven_\gn, R) \mid \gn\in\cn^\even \}  \\
&=&  \min \{\ \ind(\esodd_\gn, \Sel_{\sel(\gn)}) \mid \gn\in\cn^\odd \}. 
\end{eqnarray*}
\end{Thm}

\begin{proof}
For a vertex $v=v(\gn)$ of the graph $\xgraph$ 
of \S \ref{sheaf section}, we define $s(v)\in \mathrm{ES}(v)$ by
$$
s(v)=\left\{\begin{array}{ll} \eseven_\gn&\mathrm{if\ }\gn\in\cn^\even\\
\esodd_\gn&\mathrm{if\ }\gn\in\cn^\odd.
\end{array}\right.
$$
For an edge $e=e(\gn,\gn\gl)$ define $s(e)\in \mathrm{ES}(e)$ by
$$
s(e)=\left\{\begin{array}{ll} \loc_\gl(\esodd_\gn)
&\mathrm{if\ }\gn\in\cn^\odd\\
\loc_\gl(\esodd_{\gn\gl})&\mathrm{if\ }\gn\in\cn^\even.
\end{array}\right.
$$
The reciprocity laws of Definition \ref{es} now say exactly that,
modifying the vertex-to-edge maps (\ref{edge choice}) by an element
of $R^\times$ if needed,
the function $v\mapsto s(v)$ forms a global section of the Euler system sheaf
$\mathrm{ES}(\xgraph)$ with edge germ $e\mapsto s(e)$.
By Theorem \ref{esb}, this global section is actually a global
section of the subsheaf $\Stub(\xgraph)\subset\mathrm{ES}(\xgraph)$.
By Corollary \ref{section defect}, 
there is a unique $0\le\delta<\len(R)$
such that  $s(v)$ generates $\gm^\delta\cdot \Stub(v)$
for every vertex $v$.
For any vertex $v$ with $s(v)\not=0$ we have
$$
\delta=\ind( s(v), \Stub(v) )\le \ind( s(v), \mathrm{ES}(v) )
$$
with equality if and only if $v$ is a core vertex.  Since there are
even core vertices (by Proposition \ref{hub}),
$$
\delta=\min \{\ \ind(s(v),\mathrm{ES}(v)) \mid v \mathrm{\ even} \},
$$
and similarly with even replaced by odd.
\end{proof}

%%%%%%%%%%%%%%%%%%%%%%%%%%%%%%%%%%%%%%%%%%%%%%%%%%%%%%%%%%%%%%%

\subsection{A variant}
\label{variant}

%%%%%%%%%%%%%%%%%%%%%%%%%%%%%%%%%%%%%%%%%%%%%%%%%%%%%%%%%%%%%

In the applications to Iwasawa theory
we will need to work under slightly different
hypothesis on $T$.  In this subsection we assume that $K$ is 
a quadratic imaginary field.  Fix an embedding $K^\alg\hookrightarrow \C$
and let $\tau\in G_K$ be the associated complex conjugation.
Let $R$ and $T$ be as in the introduction to \S \ref{S:Euler Systems},
but instead of assuming that $T$ is self Cartier dual via 
an alternating pairing, assume, as in \S 1.3 of \cite{me}, 
that there is a perfect \emph{symmetric} pairing
$$
(\ ,\ ):T\times T\map{}R(1)
$$
which satisfies $(x^\sigma, y^{\tau\sigma\tau})=(x,y)^\sigma$ for 
any $\sigma\in G_K$.
Let $\Tw(T)$ be the $G_K$-module whose underlying $R$-module is $T$,
but with the $G_K$-action twisted by conjugation by $\tau$.
The above pairing can then be viewed as a perfect $G_K$-invariant
pairing 
$$
T\times\Tw(T)\map{}R(1).
$$
There is a canonical isomorphism 
$H^1(K,T)\iso H^1(K,\Tw(T))$
given on cocycles by $c(\sigma)\mapsto c^*(\sigma)=c(\tau\sigma\tau)$.
For any finite place $v$ of $K$, there is similarly a canonical isomorphism
from the local cohomology of $T$ at $v$ to the local cohomology of
$\Tw(T)$ at $\tau(v)$.  This isomorphism induces a local Tate pairing
\begin{equation}\label{twisted local duality}
H^1(K_v,T)\times H^1(K_{\tau(v)},T)\map{}R,
\end{equation}
and by direct calculation on cocycles one can check that if $v=\tau(v)$
then this pairing is symmetric.  Thus locally at a degree two prime
of $K$, the cohomology of $T$ behaves exactly as if $T$ were self-dual
via an alternating pairing.
We now define a Selmer structure $(\sel,\Sigma_\sel)$ exactly
as in \S \ref{selmer modules}, and say that a Selmer structure is
self-dual if the local conditions 
$H^1_\sel(K_v,T)$ and $H^1_\sel(K_{\tau(v)},T)$
are exact orthogonal complements under the pairing 
(\ref{twisted local duality}) for every $v\in\Sigma_\sel$.

All of the results of \S \ref{S:Euler Systems} hold verbatim under these
modified hypothesis (one need only verify that Lemma \ref{local freeness} and
Propositions \ref{structure} and \ref{global duality} hold, 
as these are the only
places where the self-duality hypotheses on $T$ and $\sel$ are 
directly invoked; for the latter two, 
the reader may consult Sections 1.4 and 1.5
of \cite{me}) with one minor caveat: the 
statement of Lemma \ref{local freeness} and the proof of  
Proposition \ref{global duality} require the self-duality of 
$H^1(K_\gl,T)$, and so we must add the hypothesis
that $\cl$ contains only degree two primes of $K$.

Finally, we remark that if the action of $G_K$ on $T$ extends to an
action of $G_\Q$ then the alternate hypotheses of this subsection are
equivalent to those of the introduction to \S \ref{S:Euler Systems},
since one may identify $T\iso\Tw(T)$ as 
$G_K$-modules via the map $x\mapsto x^\tau$.

%%%%%%%%%%%%%%%%%%%%%%%%%%%%%%%%%%%%%%%%%%%%%%%%%%%%%%%%%%%%%%%%%
%%%%%%%%%%%%%%%%%%%%%%%%%%%%%%%%%%%%%%%%%%%%%%%%%%%%%%%%%%%%%%%%%

\section{Iwasawa theory of elliptic curves}
\label{Iwasawa}

%%%%%%%%%%%%%%%%%%%%%%%%%%%%%%%%%%%%%%%%%%%%%%%%%%%%%%%%%%%%%%%%%%
%%%%%%%%%%%%%%%%%%%%%%%%%%%%%%%%%%%%%%%%%%%%%%%%%%%%%%%%%%%%%%%%%%

Let $K$ be a quadratic imaginary field of discriminant $d_K$ and
quadratic character $\epsilon$,
$p>3$ a rational prime, and $E/\Q$ an elliptic curve with conductor
$N$. Assume that $E$ has either multiplicative or good ordinary
reduction at $p$, and that $(d_K,pN)=1$.
Let $N^-$ be the largest divisor of $N$ which is prime to $p$ and satisfies
$\epsilon(q)=1$ for all primes $q\mid N^-$.
Factor $N=N^+ N^-$.
Let $\tau$ be a fixed choice of complex conjugation.

\begin{Hyp}\label{irreducible hyp}
Throughout \S \ref{Iwasawa} we assume:
\begin{enumerate}
\item $E[p]$ is absolutely irreducible as a $G_K=\Gal(K^\alg/K)$-module,
\item $N^-$ is squarefree.
\end{enumerate}
\end{Hyp}

We denote by $D_\infty$ the anticyclotomic 
$\Z_p$-extension of $K$, characterized by $\tau\sigma\tau=\sigma^{-1}$
for any $\sigma\in\Gamma=\Gal(D_\infty/K)$.
Let $D_m\subset D_\infty$ be the subfield with $[D_m:K]=p^m$,
and set $\Lambda=\Z_p[[\Gamma]]$.

%%%%%%%%%%%%%%%%%%%%%%%%%%%%%%%%%%%%%%%%%%%%%%%%%%%%%%%%%%%%

\subsection{Selmer modules over $\Lambda$}

%%%%%%%%%%%%%%%%%%%%%%%%%%%%%%%%%%%%%%%%%%%%%%%%%%%%%%%%%%%

\begin{Def}
A degree two prime $\gl\nmid N$ of $K$ is \emph{$k$-admissible} if
$\N(\gl)\not\equiv 1\pmod{p}$, and if there is a decomposition
$$
E[p^k]\iso (\Z/p^k\Z)\oplus \mu_{p^k}
$$
of $\Gal(K^\unr_\gl/K_\gl)$-modules.
A $1$-admissible prime will simply be called \emph{admissible}.
The set of $k$-admissible primes is denoted $\cl_k$, and we let
$\cn_k$ denote the set of squarefree products of primes in $\cl_k$.
\end{Def}

Let  $q\mid N^-$ be a rational prime.  By Hypothesis \ref{irreducible hyp}(b)
$E$ has multiplicative reduction at $q$, and hence split multiplicative
reduction at the prime $\gq$ of $K$ above $q$.
The Tate parametrization shows that $T_p(E)$ has the form 
$\left(\begin{matrix}\epsilon_{\mathrm{cyc}}& * \\ 0& 
1 \end{matrix}\right)$
as a $G_{K_\gq}$-module, and 
we denote by $\Fil_q(T_p(E))\subset T_p(E)$ the $\Z_p$-line on which 
$G_{K_\gq}$ acts via $\epsilon_{\mathrm{cyc}}$.
For any extension $L/K_\gq$ the \emph{ordinary} submodule
$$
H^1_\ord(L,T_p(E))\subset H^1(L,T_p(E))
$$
is defined to be the image of 
$H^1(L,\Fil_q(T_p(E)))$,
and $H^1_\ord(L,E[p^k])$ is defined similarly.
For a $k$-admissible prime $\gl\in\cl_k$ we have a similar ordinary
local condition $H^1_\ord(L,E[p^k])$ for any extension $L/K_\gl$, as
in \S \ref{ordinary selmer}.
For the prime $p$, $T_p(E)$ has a distinguished line on which 
an inertia group at $p$ in $G_\Q$ acts via the cyclotomic character.  
Call this line $\Fil_p(T_p(E))$ and define the 
ordinary condition at $p$ to be the image of 
$$
H^1(L,\Fil_p(T_p(E)))\map{}H^1(L,T_p(E))
$$
for any finite extension $L/\Q_p$, and similarly for $E[p^k]$ and 
$E[p^\infty]$.

\begin{Lem}
For any $\gl\in\cl_k$ the module
$$
\mil_m H^1_\unr(D_{m,\gl}, E[p^k]) \define 
\mil_m \bigoplus_{w\mid\gl} H^1_\unr(D_{m,w}, E[p^k])
$$
is free of rank one over $\Lambda/p^k\Lambda$, and the same is true
with $\unr$ replaced by $\ord$.
\end{Lem}

\begin{proof}
Since $\gl$ splits completely in $D_\infty$, Shapiro's lemma gives
an isomorphism
$$
\mil_m \bigoplus_{w\mid\gl} H^1(D_{m,w}, E[p^k])
\iso  H^1(K_\gl, E[p^k]\otimes\Lambda) \iso H^1(K_\gl, E[p^k])\otimes\Lambda.
$$
This, together with Lemma \ref{local freeness}, gives the claim.
\end{proof}

We define the Selmer groups 
$$
\mathcal{S}(D_\infty,T_p(E))  \subset  \mil H^1(D_m,T_p(E)) 
\hspace{1cm}
\Sel(D_\infty , E[p^\infty])  \subset  \dlim H^1(D_m,E[p^\infty])
$$
to be the classes which are ordinary at the primes dividing $pN^-$
and unramified at all other primes, and abbreviate
$$
\mathcal{S}=\mathcal{S}(D_\infty,T_p(E))
\hspace{1cm}
X=\Hom\big(\Sel(D_\infty , E[p^\infty]),\Q_p/\Z_p\big).
$$
For any $\gn\in\cn_k$, let  
$$
\mathcal{S}_\gn(D_\infty, E[p^k])\subset \mil_m H^1(D_m, E[p^k])
$$
be the $\Lambda$-submodule of classes which are ordinary at the 
primes dividing $\gn pN^-$, and unramified at all other primes.

%%%%%%%%%%%%%%%%%%%%%%%%%%%%%%%%%%%%%%%%%%%%%%%%%%%%%%%%%%%%%%

\subsection{Euler systems over $\Lambda$}
\label{es lambda section}

%%%%%%%%%%%%%%%%%%%%%%%%%%%%%%%%%%%%%%%%%%%%%%%%%%%%%%%%%%%%%%%

\begin{Def}
Given $\gn\in\cn_1$, let $n$ be the positive integer satisfying $n\co_K=\gn$.
We say that $\gn$ is \emph{definite} if $\epsilon(nN^-)=-1$,
and is \emph{indefinite} if $\epsilon(nN^-)=1$. Let 
$\cn^\definite_k\subset \cn_k$ be the subset of definite products, and
define $\cn^\indefinite_k$ similarly.
\end{Def}

Suppose that for every $k>0$ we are given families
\begin{equation}\label{lambda es}
\{\esodd_\gn\in \mathcal{S}_\gn(D_\infty, E[p^k]) 
\mid \gn\in\cn_k^\indefinite \}
\hspace{1cm}
\{\eseven_\gn\in \Lambda/p^k\Lambda 
\mid \gn\in\cn_k^\definite \}
\end{equation}
which, as $k$ varies, are compatible with the 
inclusion $\cn_{k+1}\subset\cn_k$
and the natural maps $\Lambda/p^{k+1}\Lambda\map{}\Lambda/p^k\Lambda$
and $E[p^{k+1}]\map{p}E[p^k]$. 
Assume that these classes
satisfy the first and second reciprocity laws:
\begin{enumerate}
\item for any $\gn\gl\in \cn^\indefinite_k$ there is an isomorphism
of $\Lambda$-modules 
$$
\mil_m H^1_\ord(D_{m,\gl}, E[p^k])\iso \Lambda/p^k\Lambda
$$
taking $\loc_\gl(\esodd_{\gn\gl})$ to $\eseven_\gn$;
\item for any $\gn\gl\in \cn^\definite_k$ there is an isomorphism
of $\Lambda$-modules 
$$
\mil_m H^1_\unr(D_{m,\gl}, E[p^k])\iso \Lambda/p^k\Lambda
$$
taking $\loc_\gl(\esodd_{\gn})$ to $\eseven_{\gn\gl}$.
\end{enumerate}
Since the empty product lies in $\cn_k$
for every $k$, we obtain a distinguished element
\begin{equation}\label{senator}
\begin{array}{cll}
{\eseven^{\infty}}\in\Lambda &   \mathrm{if} & \epsilon(N^-)=-1\\
{\esodd^\infty}\in\mathcal{S} &  
\mathrm{if } & \epsilon(N^-)=1
\end{array}
\end{equation}
defined as the inverse limit of $\eseven_1$ or $\esodd_1$ as $k$ varies.

\begin{Lem}\label{torsion-free}
The $\Lambda$-module $\mathcal{S}$ 
is torsion free.
\end{Lem}
 
\begin{proof}
As the torsion subgroup
of $E(D_\infty)$ is finite (since $D_\infty$ has primes of finite residue
degree), $H^0(D_\infty, T_p(E))=0$ and the claim follows from
\cite[Lemma 1.3.3]{pr95}.
\end{proof}

The following theorem will be proved in \S \ref{mc proof}.

\begin{Thm}\label{abstract mc}
Assume that the special element (\ref{senator}) is nonzero and
let $X_{\Lambda-\mathrm{tors}}$
denote the torsion submodule of $X$.
\begin{enumerate}
\item 
One has the rank formulas
$$
\mathrm{rank}_\Lambda\mathcal{S}=\mathrm{rank}_\Lambda X
=\left\{\begin{array}{ll} 0 & \mathrm{if\ }\epsilon(N^-)=-1\\
1 & \mathrm{if\ }\epsilon(N^-)=1.
\end{array}\right.
$$
\item
For any height one prime $\gp$ of $\Lambda$ one has
\begin{equation*}
\ord_\gp\big(\mathrm{char}(X_{\Lambda-\mathrm{tors}})\big) \le
2\cdot \left\{\begin{array}{ll} \ord_\gp(\eseven^\infty) 
& \mathrm{if\ }\epsilon(N^-)=-1\\
\ord_\gp\big(\mathrm{char}(\mathcal{S}/\Lambda\esodd^\infty)\big) 
& \mathrm{if\ }\epsilon(N^-)=1.
\end{array}\right.
\end{equation*}

\item 
Equality holds in (b) if the following condition is satisfied:
there exists a $k_0$ such that for all
$j\ge k_0$ the set 
$$
\{\eseven_\gn\in \Lambda/p^j\Lambda \mid \gn\in \cn_j^\definite \}
$$
contains an element with nontrivial image in $\Lambda/(\gp,p^{k_0})$.
\end{enumerate}
\end{Thm}

%%%%%%%%%%%%%%%%%%%%%%%%%%%%%%%%%%%%%%%%%%%%%%%%%%%%%%%%%%%%%%%

\subsection{Reduction at a height one prime}

%%%%%%%%%%%%%%%%%%%%%%%%%%%%%%%%%%%%%%%%%%%%%%%%%%%%%%%%%%%%%%%

Set $V_p(E)=T_p(E)\otimes\Q_p$, so that we have the exact sequence
\begin{equation}\label{short exact}
0\map{}T_p(E)\map{}V_p(E)\map{}E[p^\infty]\map{}0.
\end{equation}
Fix $\gp\not=p\Lambda$ a height-one prime of $\Lambda$, 
and denote by $\co_\gp$ the integral closure of $\Lambda/\gp$, 
viewed as a Galois module with 
trivial action. The ring $\co_\gp$ is the ring of
integers of a finite extension $\Phi_\gp/\Q_p$, and we denote by
$\gm_\gp$ its maximal ideal.
By tensoring (\ref{short exact}) with $\co_\gp$ (viewed as a $G_K$-module
via the natural map $G_K\map{}\Lambda^\times$),
we obtain an exact sequence of $\co_\gp[[G_K]]$-modules
\begin{equation}\label{twisted short exact}
0\map{}T_\gp\map{}V_\gp\map{}W_\gp\map{}0.
\end{equation}

For any prime $\gq$ of $K$ and $A$ and one of $T_\gp$, $V_\gp$, or $W_\gp$,
 we define a submodule
$$
H^1_{\sel_\gp}(K_\gq,M)\subset H^1(K_\gq,M)
$$
as follows.  First suppose $M=V_\gp$. If $\gq\nmid pN^-$ then
$H^1_{\sel_\gp}(K_\gq,V_\gp)$ is the unramified cohomology classes.
If $\gq\mid pN^-$ then $H^1_{\sel_\gp}(K_\gq,V_\gp)$ is defined to be
the image of 
$$
H^1(K_\gq, \Fil_\gq(T_p(E))\otimes \Phi_\gp)\map{} H^1(K_\gq,V_\gp).
$$
If $M=T_\gp$ or $W_\gp$, then $H^1_{\sel_\gp}(K_\gq,M)$ is obtained 
from $H^1_{\sel_\gp}(K_\gq,V_\gp)$ by propagation, in the sense of
Remark \ref{propagation}.
These local submodules define global Selmer groups which we denote
by $\Sel_{\sel_\gp}(K,M)$.

\begin{Prop}\label{control}
Shapiro's lemma,  the natural map 
$T_p(E)\otimes\Lambda\map{}T_\gp$, and its dual induce maps
\begin{eqnarray*}
\mathcal{S}/\gp\mathcal{S} \to \Sel_{\sel_\gp}(K,T_\gp)
\hspace{1cm}
\Sel_{\sel_\gp}(K,W_\gp) \to \Sel(D_\infty, E[p^\infty])[\gp].
\end{eqnarray*}
The first map is injective.
There is a finite set of height one primes $\Sigma_\Lambda$ of $\Lambda$
such that if $\gp\not\in\Sigma_\Lambda$, then these maps have
finite kernel and cokernel which are bounded by  a constant depending
on $[\co_\gp:\Lambda/\gp]$ but not on $\gp$ itself.
\end{Prop}

\begin{proof}
The proof requires only minor modifications from that of 
\cite[Proposition 5.3.14]{mazur-rubin}.
One must first prove a local control theorem at
each prime $\gq$ of $K$.  If $\gq\nmid pN^-$ this local result is
\cite[Lemma 5.3.13]{mazur-rubin}. If $\gq\mid p$ the desired result is
\cite[Lemma 2.2.7]{me}. The case $\gq\mid N^-$ is similar to the latter,
but is greatly simplified by the fact that such 
$\gq$ split completely in $D_\infty$.
With the local control results in hand, the remainder of the proof follows
that of \cite[Proposition 5.3.14]{mazur-rubin} verbatim.
\end{proof}

\begin{Lem}\label{injective reduction}
Abbreviate $\mathcal{S}_\gp= \Sel_{\sel_\gp}(K,T_\gp)$.
The natural map
$$
\mathcal{S}_\gp/p^k\mathcal{S}_\gp\map{}\Sel_{\sel_\gp}(K,T_\gp/p^k T_\gp)
$$  
is injective, where the Selmer structure on $T_\gp/p^k T_\gp$ 
is obtained from the Selmer structure on $T_\gp$ by propagation 
(Remark \ref{propagation}).
\end{Lem}

\begin{proof}
This is Lemma 3.7.1 of \cite{mazur-rubin}.
\end{proof}

For any pair of positive integers $k\le j$, set
$$
\delta_\gp(k,j)=
\min\{ \ind(\eseven_\gn, \co_\gp/p^k\co_\gp) \mid \gn\in 
\cn_j^\definite\}\le \infty.
$$ 
As $\delta_\gp(k,j)\le\delta_\gp(k,j+1)$, we may define 
$\delta_\gp(k)=\lim_{j\to\infty}\delta_\gp(k,j)$.

\begin{Prop}\label{dvr bound}
If $\epsilon(N^-)=-1$ and $\eseven^\infty\in\Lambda$ has
nontrivial image in $\co_\gp/p^k\co_\gp$ then 
$$
\len_{\co_\gp}\big(\Sel_{\sel_\gp}(K,W_\gp)\big)
+2\delta_\gp(k)= 2\cdot \len_{\co_\gp}(\co_\gp/\co_\gp\eseven^\infty).
$$
If $\epsilon(N^-)=1$ and $\esodd^\infty$ has nontrivial image in
$\mathcal{S}_\gp/p^k\mathcal{S}_\gp$ then
\begin{enumerate}
\item $\mathcal{S}_\gp$ is free of rank one over $\co_\gp$,
\item $\Sel_{\sel_\gp}(K,W_\gp)$ has $\co_\gp$-corank one, and
\item 
$
\len_{\co_\gp}\big(\Sel_{\sel_\gp}(K,W_\gp)_{/\mathrm{div}}\big)
+2\delta_\gp(k)=
2\cdot\len_{\co_\gp}\big(\mathcal{S}_\gp/
\co_\gp\esodd^\infty\big)
$
(the subscript $/\mathrm{div}$ indicates the quotient by the
maximal $\co_\gp$-divisible submodule).
\end{enumerate}
\end{Prop}

\begin{proof}
Let $k$ be as in the statement of the proposition.
For any $j\ge k$, abbreviate
$$
T_j=T_{\gp}/p^j T_\gp \hspace{1cm}
R_j=\co_{\gp}/p^j\co_\gp.
$$
Let $\sel$ denote the Selmer structure on $T_j$
obtained by propagation (Remark \ref{propagation}) of $\sel_\gp$
from $T_\gp$, and use the same notation for the Selmer structure
on $W_\gp[p^j]$ propagated from $\sel_\gp$ on $W_\gp$.
By applying the reduction maps $\Lambda/p^j \map{} R_j$ and
$$
\mil_m H^1(D_m, E[p^j])\iso H^1(K, E[p^j]\otimes\Lambda)
\map{}H^1(K,T_j)
$$
to the Euler system (\ref{lambda es}), we obtain families
\begin{equation}\label{reduced euler system}
\{\overline{\esodd}_{\gn} \in 
\Sel_{\sel(\gn)}(K,T_j)\mid \gn\in \cn_j^\indefinite\}
\hspace{.7cm}
\{\overline{\eseven}_{\gn} \in R_j \mid \gn\in \cn_j^\definite\}
\end{equation}
By assumption (and Lemma \ref{injective reduction}),
$\overline{\esodd}_{1}$ or $\overline{\eseven}_{1}$ 
(depending on whether 
$\epsilon(N^-)=1$ or $-1$) is nontrivial, where $1\in\cn_j$ is the
empty product.

A choice of uniformizer of $\co_\gp$ determines an isomorphism 
$T_j\iso W_\gp[p^j]$, and under such an isomorphism the
Selmer structures $\sel$ are identified (as in the proof of Lemma 1.3.8(i)
of \cite{rubin}).  In particular
\begin{equation}\label{important identification}
\Sel_{\sel}(K,T_j)\iso\Sel_{\sel}(K,W_\gp[p^j])
\iso \Sel_{\sel_\gp}(K,W_\gp)[p^j],
\end{equation}
where the second isomorphism follows from Lemma \ref{subquotients}
and the following

\begin{Lem}\label{first hypotheses}
The triple $(T_j,\sel,\cl_{j})$ satisfies
Hypotheses  \ref{cartesian} and \ref{useful primes}, as well as 
the hypotheses of \S \ref{variant}.  
More precisely the following hold.
\begin{enumerate}
\item $T_j$ is residually an absolutely irreducible $G_K$-module.
\item For any $\gl\in\cl_j$, the Frobenius at $\gl$ acts on 
$T_j$ with eigenvalues $\N(\gl)$ and $1$.  
\item There is a perfect $R_j$-bilinear symmetric pairing
$T_j\times T_j\map{}R_j(1)$
satisfying $(s^\sigma,t^{\tau\sigma\tau})=(s,t)^\sigma$
for any $\sigma\in G_K$.
\item The Selmer structure $\sel$ is
cartesian in the sense of Definition \ref{cartesian def}, and is self-dual
in the sense of \S \ref{variant}, relative to the pairing
above.
\item 
The set $\cl_{j}$ satisfies Hypothesis \ref{useful primes}.
\end{enumerate}
\end{Lem}

\begin{proof}
Since $\gl$ splits completely in $D_\infty$, there is an 
isomorphism of Galois modules
$T_\gp \iso T_p(E)\otimes \co_\gp$ with $G_{K_\gl}$
acting \emph{trivially} on $\co_\gp$.   
In particular, the residual representation of $T_\gp$ 
is absolutely irreducible since $E[p]$ is (by Hypothesis 
\ref{irreducible hyp}), and property (b) is immediate from the definition
of a $k$-admissible prime.
Define an $\co_\gp(1)$-valued pairing on $T_\gp$ by the rule
$$(x\otimes\alpha,y\otimes\beta)_\gp=\alpha\beta\cdot (x,y^\tau),$$
where $(\ ,\ )$ is the Weil pairing on $T_p(E)$.  The reduction
of this pairing modulo $p^j$ defines the pairing of (c).
The cartesian property of (d) is a consequence of the
fact that the Selmer structure $\sel_\gp$ on $T_\gp$ is obtained 
by propagation from $V_\gp$; see \cite[Lemma 3.7.1]{mazur-rubin}.
The self-duality follows from this and the self-duality of the local
conditions defining the canonical Selmer structure on $V_\gp$. 
Part (e) is \cite[Theorem 3.2]{BD03}.
\end{proof}

\begin{Lem}
The decomposition 
$\cn_j=\cn_j^\odd\sqcup\cn_j^\even$ (relative to the data 
$T$, $\sel$, $\cl_j$) 
of Definition \ref{stub} is given by 
\begin{equation}\label{parity decomp}
\cn_j^\odd=\cn_j^\indefinite\hspace{1cm}\cn_j^\even=\cn_j^\definite.
\end{equation}
Furthermore, the families (\ref{reduced euler system}) 
form an Euler system of odd type for $(T_j,\sel,\cl_{j})$.
\end{Lem}

\begin{proof}
First note that either (\ref{parity decomp}) holds or the opposite relation
$$
\cn_j^\even=\cn_j^\indefinite\hspace{1cm}\cn_j^\odd=\cn_j^\definite
$$
holds (simply because the even/odd decomposition of $\cn_j$
is determined by the function $\rho(\gn)$ of Corollary \ref{rho}, the 
definite/indefinite decomposition is determined by $\epsilon(\gn)$,
and both functions are multiplied by $-1$ when one replaces $\gn$
by $\gn\gl$.)  The reciprocity laws of \S \ref{es lambda section}
imply that the reduced families (\ref{reduced euler system})
satisfy the reciprocity laws of Definition \ref{es},
and so this family forms a nonzero Euler system which is of odd
type if (\ref{parity decomp}) holds, and is of even type otherwise.
By Proposition \ref{no even es}, the Euler system must by
of odd type, so (\ref{parity decomp}) holds.
\end{proof}

The Euler system of the lemma for $(T_k, \sel, \cl_k)$ may not be free,
but this can be remedied by shrinking  the set of indexing primes $\cl_k$
slightly.

\begin{Lem}\label{free es}
For any $j\ge 2k$ the families
$$
\{\overline{\esodd}_{\gn} \in 
\Sel_{\sel(\gn)}(K,T_k)\mid \gn\in \cn_j^\indefinite\}
\hspace{.7cm}
\{\overline{\eseven}_{\gn} \in R_k \mid \gn\in \cn_j^\definite\}
$$
form a free Euler system of odd type for $(T_k,\sel, \cl_j)$.
\end{Lem}

\begin{proof}
Fix $\gn\in\cn_j^\odd=\cn_j^\indefinite$ and $j\ge 2k$.  We must show that
there is a free rank one $R_k$-submodule of $\Sel_{\sel(\gn)}(K,T_k)$
containing $\overline{\esodd}_\gn$.  By Proposition \ref{structure}
we may decompose 
$$
\Sel_{\sel(\gn)}(K,T_j)\iso R_j\oplus N\oplus N
\hspace{1cm}
 \Sel_{\sel(\gn)}(K,T_k)\iso R_k\oplus M\oplus M.
$$
Let $\gm$ be the maximal ideal of $\co_\gp$ and fix a uniformizer $\pi$.
Let $e$ be the ramification degree of $\co_\gp$, so that $R_k$ has length
$ek$.  If $\gm^{ek-1}M\not=0$ then $\overline{\esodd}_\gn=0$
by Proposition \ref{annihilation}, and there is nothing to prove.
Assume therefore that $\gm^{ek-1}M=0$.
Lemma \ref{subquotients} gives a commutative diagram
$$
\xymatrix{
{\Sel_{\sel(\gn)}(K,T_j) \ar[d]\ar[rr]^{\pi^{e(j-k)}}} 
& & {\Sel_{\sel(\gn)}(K,T_j)[\gm^{ek}]} \\ 
{\Sel_{\sel(\gn)}(K,T_k). \ar[rru]^\iso} 
}
$$
The diagonal isomorphism implies that 
$\gm^{ek-1}\cdot \Sel_{\sel(\gn)}(K,T_j)[\gm^{ek}]$ 
is a cyclic module, and so
$\gm^{ek-1}N=0$.  But $j\ge 2k$ then implies that the image of $N$ under the
vertical arrow is zero, and hence the image of the vertical arrow
is free of rank one.  Since $\overline{\esodd}_\gn$ is contained in this
image, the claim is proved.
\end{proof}

Now fix $j\ge 2k$.
Since $\cn_j^\even=\cn_j^\definite$, the empty product lies in $\cn_j^\even$
if and only if $\epsilon(N^-)=-1$.  If this is the case, then
applying Theorem \ref{rigidity} with $\gn=1$ tells us that
$\Sel_{\sel}(K,T_k)\iso M\oplus M$
with 
$$
\len_{\co_\gp}(M)+\delta_\gp(k,j) = \ind(\overline{\eseven}_1, R_k)
=\ind(\eseven^\infty, \co_\gp/p^k\co_\gp).
$$
In particular, since the right hand side is $<k$,
 (\ref{important identification}) implies that
$M\oplus M\iso \Sel_{\sel_\gp}(K,W_\gp)$. We conclude
$$
\len_{\co_\gp}(\Sel_{\sel_\gp}(K,W_\gp))+2\cdot\delta_\gp(k,j)=
2\cdot \len_{\co_\gp}(\co_\gp/\co_\gp\eseven^\infty).
$$

Now consider the case $\epsilon(N^-)=1$.  First note that
Theorem \ref{rigidity} (again with $\gn=1$) tells us that
$\Sel_{\sel}(K,T_k)\iso R\oplus M\oplus M$
with 
$$
\len_{\co_\gp}(M) +\delta_\gp(k,j)= 
\ind\big(\overline{\esodd}_1, \Sel_\sel(K,T_k)\big).
$$
As above, this implies that $\len_{\co_\gp}(M)<k$.
Combining this with (\ref{important identification}) tells us that
$
\mathcal{S}_\gp\iso\mil_k \Sel_{\sel_\gp}(K,W_\gp)[p^k]
$
is a torsion-free rank-one $\co_\gp$-module.  
By  \cite[Lemma 3.7.1]{mazur-rubin} the reduction map
$$
\mathcal{S}_\gp/p^k\mathcal{S}_\gp\map{}\Sel_\sel(K,T_k)
$$  
is injective, and it follows  from Theorem \ref{rigidity} that 
\begin{eqnarray*}
\len_{\co_\gp}\big(\Sel_{\sel_\gp}(K,W_\gp)_{/\mathrm{div}}\big)
+2\delta_\gp(k,j)
&=& \len_{\co_\gp}(M\oplus M) +2\delta_\gp(k,j)\\
&=& 2\cdot \ind(\overline{\esodd}_1, \Sel_\sel(K,T_k)) \\
&=&2\cdot \len_{\co_\gp}(\mathcal{S}_\gp/\mathcal{S}_\gp\esodd^\infty\big).
\end{eqnarray*}
Now take $j\to\infty$. This completes the proof of Proposition \ref{dvr bound}.
\end{proof}

%%%%%%%%%%%%%%%%%%%%%%%%%%%%%%%%%%%%%%%%%%%%%%%%%%%%%%%%%%%%%%%%%%%%%%

\subsection{Proof of Theorem \ref{abstract mc}}
\label{mc proof}

%%%%%%%%%%%%%%%%%%%%%%%%%%%%%%%%%%%%%%%%%%%%%%%%%%%%%%%%%%%%%%%%%%%%%%

The theorem is reduced to Proposition \ref{dvr bound} exactly as in 
the proof of Theorem 5.3.10 of \cite{mazur-rubin}.

Assume that $\esodd^\infty$ or $\eseven^\infty$ is nonzero,
depending on whether we are in the case $\epsilon(N^-)=1$ or $-1$.
Since $\mathcal{S}$ is a finitely generated torsion-free $\Lambda$-module
(Lemma \ref{torsion-free}), if $\epsilon(N^-)=1$
it is easily seen that the image
of $\esodd^\infty$ in $\mathcal{S}/\gp\mathcal{S}$ is nonzero
for all but finitely many height-one primes $\gp$. Similar comments
hold for $\eseven^\infty$ when $\epsilon(N^-)=-1$.
Fix a finite set $\Sigma_\Lambda$ of height one primes of $\Lambda$
as in Proposition \ref{control} large enough that $\Sigma_\Lambda$
contains $p\Lambda$ and all prime divisors of the characteristic ideal 
of the torsion submodule of $X$, and
large enough that the special element (\ref{senator}) has nonzero image in 
$\mathcal{S}/\gp\mathcal{S}$ or $\Lambda/\gp\Lambda$ 
for all $\gp\not\in\Sigma_\Lambda$.

Fix any $\gp\not\in\Sigma_\Lambda$ and suppose $\epsilon(N^-)=1$.
By Proposition \ref{control}, $\kappa^\infty$ has nonzero image in
$\Sel_{\sel_\gp}(K,T_\gp)$. Proposition \ref{dvr bound} then implies that
$\Sel_{\sel_\gp}(K,T_\gp)$ and $\Sel_{\sel_\gp}(K,W_\gp)$ have rank and corank
one (respectively) as $\co_\gp$-modules.  It now follows from Proposition 
\ref{control} that 
$$
\mathrm{rank}_\Lambda\mathcal{S}=\mathrm{rank}_{\co_\gp}
(\mathcal{S}\otimes_\Lambda\co_\gp)=1
$$
and similarly for $X$.  The case $\epsilon(N^-)=-1$ is 
similar, and this completes the proof of (a).

Let $\gp$ be any height-one prime of $\Lambda$ different from $p\Lambda$,
and let $f\in\Lambda$ be a distinguished polynomial which generates $\gp$.
For each positive integer $m$ set $\gp_m=(f+p^m)\Lambda$.  For 
$m\gg 0$, $\gp_m$ is a prime ideal $\not\in\Sigma_\Lambda$ with
$\Lambda/\gp\iso\Lambda/\gp_m$ as rings (by Hensel's lemma).
Arguing as in the proof of Theorem 5.3.10 of \cite{mazur-rubin} and using
Proposition \ref{control}, we obtain
$$
\len_{\Z_p}\big( \Sel_{\sel_{\gp_m}}
(K, W_{\gp_m})_{/\mathrm{div}}\big)
= m\ \mathrm{rank}_{\Z_p}(\co_\gp) \cdot 
\ord_\gp\big(\mathrm{char}(X_{\Lambda-\mathrm{tors}})\big)
$$
up to $O(1)$ as $m$ varies.
Similarly, writing 
$\mathcal{S}_{\gp_m}=\Sel_{\sel_{\gp_m}}(K,T_{\gp_m})$,
\begin{eqnarray*}
\len_{\Z_p}(\mathcal{S}_{\gp_m}/\co_{\gp_m}\esodd^\infty)
&=&
m\ \mathrm{rank}_{\Z_p}(\co_\gp) \cdot 
\ord_\gp\big(\mathrm{char}(\mathcal{S}/\Lambda \esodd^\infty)\big)
\\
\len_{\Z_p}(\co_{\gp_m}/\co_{\gp_m}\eseven^\infty)
&=&
m\ \mathrm{rank}_{\Z_p}(\co_\gp) \cdot 
\ord_\gp(\eseven^\infty)
\end{eqnarray*}
when $\epsilon(N^-)=1$ or $-1$, respectively, up to $O(1)$ as
$m$ varies.
Proposition \ref{dvr bound} (with $k\gg 0$) gives the inequality
\begin{eqnarray*}
\len_{\Z_p}\big( \Sel_{\sel_{\gp_m}}(K, W_{\gp_m})_{/\mathrm{div}}\big)
+2e\delta_{\gp_m}(k)
&=& 2\cdot
\len_{\Z_p}(\mathcal{S}_{\gp_m}/\co_{\gp_m}\esodd^\infty)\\
\len_{\Z_p}\big( \Sel_{\sel_{\gp_m}}(K, W_{\gp_m})\big)
+2e\delta_{\gp_m}(k)
&=& 2\cdot
\len_{\Z_p}(\co_{\gp_m}/\co_{\gp_m}\eseven^\infty)
\end{eqnarray*}
(again, when $\epsilon(N^-)=1$ or $-1$, respectively)
where $e$ is the absolute ramification degree of $\co_{\gp_m}$,
which is independent of $m$.
As $\delta_{\gp_m}(k)\ge 0$, letting $m\to\infty$ proves the
inequality of (b) when $\gp\not=p\Lambda$.

We show that under the additional hypothesis of (c)
the value of $\delta_{\gp_m}(k)$ is bounded as $m$ and $k$ vary.
For every $j\ge k_0$ let $\gn(j)\in\cn_j^\definite$ be such that
$\lambda_{\gn(j)}$ has nonzero image in $\Lambda/(\gp,p^{k_0})$.
Then $\lambda_{\gn(j)}$ has nontrivial image in 
$\Lambda/(\gp_m,p^{k_0})$ for all $m\ge k_0$.  
Define $C_m$ to be the cokernel of $\Lambda/\gp_m\hookrightarrow\co_{\gp_m}$.
The groups  $C_m$ are finite, and up to isomorphism do not depend on $m$.  
If $k_1$ is large enough that $p^{k_1-k_0}$ kills $C_m$, then we
have the exact and commutative diagram
$$
\xymatrix{
{C_m[p^{k_1}]\ar[r]\ar[d]^0}  & {\Lambda/(\gp_m,p^{k_1})\ar[r]\ar[d]} &
{\co_{\gp_m}/p^{k_1}\co_{\gp_m}\ar[d]}  \\
{C_m[p^{k_0}]\ar[r]} & {\Lambda/(\gp_m,p^{k_0})\ar[r]}  &
{\co_{\gp_m}/p^{k_0}\co_{\gp_m}}.
}
$$
It follows that $\lambda_{\gn(j)}$ has nontrivial image in 
$\co_{\gp_m}/p^{k_1}\co_{\gp_m}$ for all $j\ge k_1$.
For $j\ge k\ge k_1$ we then have
$$
\delta_{\gp_m}(k,j)\le \ind(\lambda_{\gn(j)}, \co_{\gp_m}/p^{k}\co_{\gp_m})
< ek_1,
$$
hence $\delta_{\gp_m}(k)<ek_1$ for all $k\ge k_1$
and any $m\ge k_0$.

Finally, if $\gp=p\Lambda$, one instead takes 
$\gp_m=\big((\gamma-1)^m+p\big)\Lambda$ for some generator $\gamma\in\Gamma$
and a similar argument holds.  This completes the proof of 
Theorem \ref{abstract mc}.

\bibliographystyle{plain}

%\bibliography{bipartite.bib} 
 
\end{document}